\documentclass[11pt]{amsart}

\title{On the BCI problem}

\author{Ted Dobson \\
IAM and FAMNIT, University of Primorska\\
Muzejska trg 2\\
Koper 6000, Slovenia\\
\\
and\\
\\
Gregory Robson\\
FAMNIT, University of Primorska\\
Muzejska trg 2\\
Koper 6000, Slovenia\\
}

\usepackage{amssymb,latexsym,amsmath,amsthm}
\usepackage{graphicx}
\usepackage[T1]{fontenc}
\usepackage{hyperref}
\usepackage{cleveref}
\usepackage{pdflscape}
\usepackage{tikz-cd}
\usepackage[justification=centering]{caption}
\usepackage{mathrsfs}
\usepackage{ulem}
\usepackage[all,arc,curve,color,frame,pdf,line]{xy}
\usepackage{comment}
\usepackage{enumitem}

\usepackage{tikz}
\usetikzlibrary{graphs,graphs.standard,quotes}
\usetikzlibrary{arrows.meta}
\tikzstyle{V}=[draw, fill =black, circle, inner sep=0pt, minimum size=4pt]

\usepackage[margin=1.0in]{geometry}
\usepackage{fancyhdr}
\usepackage{lipsum}

\theoremstyle{plain}
\numberwithin{equation}{section}
\newtheorem{theorem}{Theorem}[section]
\newtheorem{corollary}[theorem]{Corollary}

\newtheorem{lemma}[theorem]{Lemma}

\theoremstyle{definition}
\newtheorem*{solution*}{Solution}
\newtheorem{definition}[theorem]{Definition}
\newtheorem{remark}[theorem]{Remark}
\newtheorem{example}[theorem]{Example}
\newtheorem{problem}[theorem]{Problem}

\DeclareMathOperator{\B}{\mathcal{B}}

\def\fix{{\rm fix}}
\def\Cay{{\rm Cay}}
\def\Aut{{\rm Aut}}
\def\Haar{{\rm Haar}}
\def\tl{\triangleleft}

\def\la{{\langle}}
\def\ra{{\rangle}}
\def\Z{{\mathbb Z}}
\def\cal{\mathcal}
\def\Stab{{\rm Stab}}

\def\Iso{{\rm Iso}} 
\def\ABCI{{\rm ABCI}} 
\def\CI{{\rm CI}}

\begin{document}

\thanks{The work of the first author is supported in part by the  Slovenian Research Agency (research program P1-0285 and research projects N1-0140, N1-0160, J1-2451, N1-0208, J1-3001, J1-3003, J1-4008, and J1-50000), while the work of the second author is supported in part by the Slovenian Research Agency (research program P1-0285 and Young Researchers Grant).}

\begin{abstract}
Let $G$ be a group.  The BCI problem asks whether two Haar graphs of $G$ are isomorphic if and only if they are isomorphic by an element of an explicit list of isomorphisms.  We first generalize this problem in a natural way and give a theoretical way to solve the isomorphism problem for the natural generalization.  We then restrict our attention to abelian groups and, with an exception, reduce the problem to the isomorphism problem for a related quotient, component, or corresponding Cayley digraph.  For Haar graphs of an abelian group of odd order with connection sets $S$ those of Cayley graphs (i.e. $S = -S$), the exception does not exist.  For Haar graphs of cyclic groups of odd order with connection sets those of a Cayley graph, among others, we solve the isomorphism problem.  
\end{abstract}

\maketitle



This is the fourth paper in an unplanned series of papers whose aim now is to investigate the relationship between symmetries of Cayley digraphs of a group $G$ and other graphs and digraphs obtained from different actions of the group $G$ (other than the regular one as for Cayley digraphs), as well as similar relationships for problems that depend on symmetries.  The first paper in this series \cite{Dobson2022}, motivated by the work in this paper, gave the relationship between the automorphism group of Cayley digraphs and Haar graphs of abelian groups.  The second \cite{BarberDpreprint}, motivated by the problem of defining generalized wreath products for Cayley digraphs of nonabelian groups, not only did this, but also gave the relationship between the automorphism groups of Cayley digraphs and double coset digraphs of the same group and connection set.  The third \cite{BarberDRpreprint} gives a recognition theorem for which bicoset digraphs can be written as a nontrivial $X$-join with empty graphs, the bipartite analogue of \cite[Theorem 3.3]{BarberDpreprint}, which is the crucial tool in giving the relationship between the automorphism groups of Cayley digraphs and their corresponding double coset digraphs.  In this paper, we consider applications of the work in \cite{Dobson2022} to the isomorphism problem for Haar graphs of abelian groups.

It is well known that the isomorphism problem for Cayley digraphs depends on a property of its automorphism group (Lemma \ref{Ciso}).  It is also known that there is a relationship between the automorphism group of a Haar graph of a group $G$ with connection set $S$ (also called a bi-Cayley graph) and the automorphism group of the Cayley digraph of a group $G$ with connection set $S$ (see Theorem \ref{automorphism main result} (\ref{third}) for this relationship for abelian groups).  The original motivation for this paper was to see what information one could obtain about the isomorphism problem for Haar graphs from information about the isomorphism problem for Cayley digraphs.  We remark that as there is much that is known (and not known) about the Cayley isomorphism problem for digraphs, this is a reasonable idea to investigate.  In \cite{Dobson2022} the relationship between the automorphism group of a Haar graph of an abelian group and its corresponding Cayley digraph was studied, motivated by hopes of using that information to obtain results on isomorphisms between Haar graphs.

That being said, the existing isomorphism problem for Haar graphs is called the BCI problem (BCI stands for bi-Cayley isomorphism).  It asks whether two Haar graphs of a group $G$ are isomorphic by a specific list $L$ of maps (as does the more famous CI problem for Cayley digraphs).  We will show that the Cayley isomorphism problem can be generalized in at least two ways, and that the list of maps of the current BCI problem is {\it not} a shortest list of natural maps to check for isomorphism, but is {\it shorter} than that (it is, though, one of the ways the CI problem can be generalized).  Maps that interchange the natural bipartition of a Haar graph are missing from $L$.  We propose an alternative BCI problem and show that the list of possible maps we propose is the shortest list of elements that normalize the natural semiregular subgroup isomorphic to a group $G$ that is contained in the automorphism group of a Haar graph of $G$ (another way in which the CI problem can be generalized).  This is the subject of Section \ref{ABCI problem}.
As we are still motivated by the problem of transferring information about isomorphisms between Cayley digraphs to isomorphisms between Haar graphs, we are not only concerned with this process for CI-digraphs of $G$ (those with the ``nicest'' possible list of isomorphisms to check), but also with the isomorphism problem for all Cayley digraphs.  For the CI problem, there are results that reduce the problem to conjugacy classes of regular subgroups isomorphic to $G$ in the automorphism group of a Cayley digraph of $G$, and so in Section \ref{ABCI extensions} we develop similar results to reduce the isomorphism problem of Haar graphs of $G$ to conjugacy classes of semi-regular subgroups of $G$ in their automorphism groups. 

In \cite{Dobson2022}, automorphism groups of Haar graphs of abelian groups are shown to be in one of four families, two of which reduce to automorphism groups of subgraphs or quotients of Haar graphs, one that reduces to the automorphism group of their corresponding Cayley digraphs, and one other family that seems difficult to handle.  In Section \ref{three families} we consider the isomorphism problem for the first three families, reducing the problem to either the isomorphism problem for related quotients, subgraphs, or Cayley digraphs.  Recently, Dave Morris has shown \cite{Morris2021} that for Haar graphs of abelian groups whose corresponding connection set gives a Cayley graph, the fourth family does not exist.  We discuss applications of that result together with the other work in this paper in Section \ref{applications}, and, among other results, we solve the isomorphism problem for Haar graphs of cyclic groups of odd order whose connection sets are those of Cayley graphs (Corollary \ref{cyclic solution}).


\section{Preliminaries}

We begin with basic definitions and notation that will be used throughout the paper.  Other, more specific terms, will be given as needed.

\begin{definition}
Let $G$ be a group, and $S \subseteq G$. Define a \textbf{Cayley digraph of $G$}, denoted $\text{Cay}(G,S)$, to be the digraph with vertex set $V(\text{Cay}(G,S)) = G$ and arc set $A(\text{Cay}(G,S)) = \{(g,gs) : g \in  \nobreak G,  s \in S\}.$ We call $S$ the \textbf{connection set of $\text{Cay}(G,S)$}.
\end{definition}

The condition that $A(\Cay(G,S)) = \{(g,gs):g\in G,s\in S\}$ is equivalent to $(x,y)\in A(\Cay(G,S))$ if and only if $x^{-1}y\in S$.

\begin{definition}\label{leftregularrepresentation}
Let $G$ be a group, and $g\in G$.  Define $g_L\colon G\to G$ by $g_L(x) = gx$.  The map $g_L$ is a {\bf left translation of $G$}.  The {\bf left regular representation of $G$}, denoted $G_L$, is $G_L = \{g_L:g\in G\}$.  That is, $G_L$ is the group of all left translations of $G$. 
\end{definition}

It is straightforward to verify that $G_L$ is a group isomorphic to $G$.  It is also straightforward to verify that $G_L\le\Aut(\Cay(G,S))$.  

\begin{definition}
Let $G$ be a group, and $S\subseteq G$.  Define the {\bf Haar graph} $\Haar(G,S)$ with {\bf connection set} $S$ to be the graph with vertex set $\Z_2\times G$ and edge set $\{(0,g)(1,gs):g\in G {\rm \ and\ }s\in S\}$.
\end{definition}

We remark that some authors refer to Haar graphs as bi-Cayley graphs, and denote them accordingly.  As with Cayley graphs, note that $(0,x)(1,y)\in E(\Haar(G,S))$ if and only if $x^{-1}y\in S$.  Also note that Haar graphs are natural bipartite analogues of Cayley digraphs and have bipartition ${\cal B} = \{\{i\}\times G:i\in\Z_2\}$.  We set $B_i = \{\{i\}\times G\}$ so that ${\cal B} = \{B_0,B_1\}$, and let $F$ be the set-wise stabilizer of $B_0$ and $B_1$ in $\Aut(\Haar(G,S))$.  This notation will be used throughout the paper. Haar graphs were introduced in \cite{HladnikMP2002}. 

\begin{definition}
For $g\in G$, define $\widehat{g}_L : \mathbb{Z}_2 \times G \rightarrow \mathbb{Z}_2 \times G$ by $\widehat{g}_L(i,j) = (i, g_L(j))$.  It is straightforward to verify that $\widehat{g}_L$ is a bijection.   Define $\widehat{G}_L = \{\widehat{g}_L:g\in G\}$.  
\end{definition}

As with Cayley digraphs, it is straightforward to verify that $\widehat{G}_L$ is a group isomorphic to $G$ and $\widehat{G}_L\le \Aut(\Haar(G,S))$.  Note that while $G_L$ is transitive, $\widehat{G}_L$ has two orbits of size $\vert G\vert$.  We now give some permutation group theoretic terms.

\begin{definition}
Let $G\le S_n$ with orbit ${\cal O}$, and $g\in G$. Then $g$ induces a permutation on ${\cal O}$ by restricting the domain of $g$ to ${\cal O}$.  We denote the resulting permutation in $S_{\cal O}$ by $g^{\cal O}$.  The group $G^{\cal O} = \{g^{\cal O}:g\in G\}$ is the {\bf transitive constituent} of $G$ on ${\cal O}$.  
\end{definition}

\begin{definition}\label{imprimitive}
Let $X$ be a set, and let $G\le S_X$ be transitive.  A subset $C\subseteq X$ is a {\bf block} of $G$ if whenever $g\in G$, then $g(C)\cap C = \emptyset$ or $C$.  If $C = \{x\}$ for some $x\in X$ or $C = X$, then $C$ is a {\bf trivial block}.  Any other block is nontrivial.  Note that if $C$ is a block of $G$, then $g(C)$ is also a block of $G$ for every $g\in G$, and is called a {\bf conjugate block of $C$}.  The set of all blocks conjugate to $C$ is a partition of $X$, called a {\bf block system of $G$}.
\end{definition}

\begin{definition}
Let $G\le S_n$ be transitive with a block system ${\cal C}$.  By $\fix_G({\cal C})$ we denote the subgroup of $G$ that fixes each
block of ${\cal C}$ set-wise.  That is, $\fix_G({\cal C}) = \{g\in G:g(C) = C{\rm\ for\ all\ }C\in{\cal C}\}$.  
We denote the induced action of $G$ on the block system ${\cal C}$ by $G/{\cal C}$, and the action of an element $g\in G$ on ${\cal C}$ by $g/{\cal C}$.  That is, $g/{\cal C}(C) = C'$ if and only if $g(C) = C'$, and $G/{\cal C} = \{g/{\cal C}:g\in G\}$.  
\end{definition}

\section{The Alternative BCI problem}\label{ABCI problem}

To begin, we will calculate $N_{S_{\Z_2\times G}}(\widehat{G}_L)$.  This has been done more or less in \cite{ZhouF2016}, but the normalizers are not explicitly stated, nor is the structure of the normalizer given.  For completeness, we will derive them here.

\begin{lemma}\label{normalizer}
Let $G$ be a group, $A = \Aut(G)$, $g\in G$, and $\sigma \in S_G$.  Define $\bar{g},\widetilde{g},\tau,\widehat{\sigma}:\Z_2\times G\to\Z_2\times G$ by $$\bar{g}(0,x) = (0,x){\rm \ \ \ and\ \ \ }\bar{g}(1,x) = (1,xg^{-1}),$$ $$\widetilde{g}(0,x) = (0,xg^{-1}){\rm \ \ \  and\ \ \ }\widetilde{g}(1,x) = (1,x),$$ $$\tau(i,j) = (i + 1,j),\ \ \ {\rm and}\ \ \ \widehat{\sigma}(i,j) = (i,\sigma(j)).$$   Set $\bar{G} = \{\bar{g}:g\in G\}$, $\widetilde{G} = \{\widetilde{g}:g\in G\}$, and $\widehat{A} = \{\widehat{a}:a\in A\}$.  Then $\widetilde{g} = \tau^{-1}\bar{g}\tau$ for all $g\in G$, $\bar{G}$, and $\widetilde{G}$ are groups with $\widetilde{G}\bar{G}\cong G\times G$ and $N_{S_{\Z_2\times G}}(\widehat{G}_L) = \la\tau\ra\cdot \widehat{A}\cdot(\widetilde{G}\bar{G})$.
\end{lemma}

\begin{proof}
Straightforward calculations will show that $\tau,\widehat{\alpha},\widetilde{g}$, and $\bar{g}$ all normalize $\widehat{G}_L$ (this was also observed in \cite{ZhouF2016}).  For the calculations for $\widetilde{g}$ and $\bar{g}$, notice that we multiply on the right and that the centralizer in $S_G$ of $G_L$ is $G_R$ \cite[Lemma 4.2A]{DixonM1996}, where $G_R$ is the right regular representation of $G$.  Then $\widetilde{G}$ and $\bar{G}$ are canonically isomorphic to $G$, and the elements of $\widetilde{G}$ and $\bar{G}$ obviously commute, so $\widetilde{G}\bar{G}\cong G\times G$.  Clearly $\tau$ normalizes $\widetilde{G}\bar{G}$ and $\widehat{A}$, and $\widetilde{g} = \tau^{-1}\bar{g}\tau$ for all $g\in G$.  It is also clear that $\widehat{A}$ normalizes $\widetilde{G}\bar{G}$, and so $\widetilde{G}\bar{G}\tl \la\tau,\widehat{A},\widetilde{G},\bar{G}\ra = M$. Any element of $\widetilde{G}\bar{G}$ is either semiregular on $B_i$ or fixes each point of $B_{i+1}$, for some $i\in \Z_2$, while $\widehat{A}$ fixes a point in each of $B_0$ and $B_1$, namely $(0,1_G)$ and $(1,1_G)$. Thus $(\widetilde{G}\bar{G})\cap \widehat{A} = 1$.   Also, $\la\tau\ra\cap \widehat{A} = 1 = \la\tau\ra\cap (\widetilde{G}\bar{G})$.  Thus $M = \la\tau\ra\cdot \widehat{A}\cdot(\widetilde{G}\bar{G})$.

As $\widehat{G}_L^{B_i}$, $i\in\Z_2$, is transitive on $B_i$, and $\tau$ interchanges $B_0$ and $B_1$, we see that $N = N_{S_{\Z_2\times G}}(\widehat{G}_L)$ is transitive.  Additionally, as $\widehat{G}_L\tl N$, $N$ has ${\cal B}$ as a normal block system. As $\tau/{\cal B}\not = 1$ and $N/{\cal B}\cong\Z_2$, it suffices to show that $F = \fix_N({\cal B}) = \widehat{A}\cdot(\widetilde G\bar{G})$.  

Let $\omega \in F$.  As $\widetilde{G}$ is transitive on $B_0$, there exists $x\in G$ such that $\delta = \omega\widetilde{x}$ fixes $(0,1_G)$.  As $\omega$ and $\widetilde{x}$ normalize $\widehat{G}_L$, so does $\delta$.  Let $g\in G$.  Hence $\delta^{-1}\widehat{g}_L\delta =\widehat{k_g}_L$, for some $k_g\in G$.  The map $\alpha:G\to G$ given by $\alpha(g) = k_g$ is then an automorphism of $G$ as conjugation induces a bijection and  $$\delta^{-1}\widehat{gh}_L\delta = \delta^{-1}\widehat{g}_L\widehat{h}_L\delta = \delta^{-1}\widehat{g}_L\delta\cdot\delta^{-1}\widehat{h}_L\delta.$$ Equivalently, $\alpha(gh) = \alpha(g)\alpha(h)$.  As $\delta^{-1}\widehat{g}_L\delta = \widehat{k_g}_L$ for $h\in G$ and $i\in\Z_2$ we have 
\begin{eqnarray*}
\widehat{\alpha}^{-1}\delta^{-1}\widehat{g}_L\delta\widehat{\alpha}(i,h) & = & \widehat{\alpha}^{-1}\widehat{k_g}_L\widehat{\alpha}(i,h) = \widehat{\alpha}^{-1}\widehat{k_g}_L(i,\alpha(h))\\
                                                             & = & \widehat{\alpha}^{-1}(i,k_g\alpha(h)) = \widehat{\alpha}^{-1}(i,k_gk_h) \\
                                                             & = & \widehat{\alpha}^{-1}(i,k_{gh}) =  (i,gh) = \widehat{g}_L(h).
\end{eqnarray*}
Thus $\widehat{\alpha}^{-1}\delta^{-1}\bar{g}_L\delta\widehat{\alpha}= \widehat{g}_L$, $\delta\widehat{\alpha}$ centralizes $\widehat{G}_L$ and fixes $(0,1_G)$.  

As mentioned above, the centralizer in $S_G$ of $G_L$ is $G_R$.  Hence $\delta\widehat{\alpha}^{B_i}$ is naturally isomorphic to an element of $G_R$, $i\in\Z_2$, and so $\delta\widehat{\alpha}\in \widetilde{G}\bar{G}$.  As $G_R$ is regular and $\delta\widehat{\alpha}$ fixes $(0,1_G)$, we see that $\delta\widehat{\alpha}\in \bar{G}$.
Thus $\omega\widetilde{x}\widehat{\alpha} = \bar{g'}$ for some $g'\in G$, and $\omega = \bar{g'}\widehat{\alpha}^{-1}\widetilde{x}^{-1}$.  As $\widehat{A}$ normalizes $\bar{G}$, there is $g\in G$ such that $\bar{g'}\widehat{\alpha}^{-1} = \widehat{\alpha}^{-1}\bar{g}$.  Hence $\omega = \widehat{\alpha}^{-1}\bar{g}\widetilde{x}^{-1} = \widehat{\alpha}^{-1}\widetilde{x}^{-1}\bar{g}$.  We conclude that $F = \widehat{A}\cdot(\widetilde{G}\bar{G})$ as required.  
\end{proof}

Observe that $\widehat{G_L} = \widehat{G}_L$ as previously defined.  We will use the notation $\widehat{G}_L$ for $\widehat{G_L}$.

\begin{remark} 
An obvious question arises from the statement of Lemma \ref{normalizer}.  Where is $\widehat{G}_L$ in $N_{S_{\Z_2\times G}}(\widehat{G}_L)$?  As mentioned above, $G_R$ centralizes and so normalizes $G_L$, and, as also mentioned above, $N_{S_G}(G_L) = \Aut(G)\cdot G_L$.  This means $G_R\le \Aut(G)\cdot G_L$.  We also have $N_{S_G}(G_R) = \Aut(G)\cdot G_R$, and $G_L$ centralizes $G_R$.  So $G_L\le\Aut(G)\cdot G_R$.  Thus $\widehat{G}_L\le\Aut(G)\cdot(\widetilde{G}\bar{G})$.
\end{remark}

We now give the basic definition regarding the Cayley isomorphism problem in its most general form.

\begin{definition}
A {\bf Cayley object} $X$ of a group $G$ in a class ${\cal K}$ of combinatorial objects is one in which $G_L\le\Aut(X)$, the automorphism group of $X$.
\end{definition}
 
Although we will not dwell on the definition of a combinatorial object, the interested reader is referred to \cite{Muzychuk1999}, where a discussion of some definitions is given.

A classical result of Sabidussi \cite[Lemma 4]{Sabidussi1958} (see \cite[Theorem 1.20]{Book} for a more modern proof) shows that a digraph $\Gamma$ is isomorphic to a Cayley digraph of a group $G$ if and only if $\Gamma$ contains a regular subgroup isomorphic to $G$.  Thus the above definition is consistent with the definition of a Cayley digraph.

\begin{definition}
Let $\Gamma$ be a Cayley (di)graph of $G$ such that if $\Gamma'$ is any Cayley (di)graph of $G$, then $\Gamma$ and $\Gamma'$ are isomorphic if and only if they are isomorphic by a group automorphism of $G$.  Such a Cayley (di)graph of $G$ is called a {\bf CI-(di)graph of $G$}.  Similarly, for a Cayley object $X$ of $G$ in some class of combinatorial objects ${\cal K}$, we say that $X$ is a {\bf CI-object of $G$} if and only if whenever $X'$ is another Cayley object of $G$ in ${\cal K}$, then $X$ and $X'$ are isomorphic if and only if $\alpha(X) = X'$ for some $\alpha\in\Aut(G)$.   A group $G$ which has the property that any two Cayley objects of $G$ in ${\cal K}$ are isomorphic if and only if they are isomorphic by a group automorphism of $G$ is called a {\bf CI-group with respect to ${\cal K}$}. 
\end{definition}


Almost all results showing that a particular group is a CI-group with respect to some class of combinatorial objects make use of the following result due to Babai \cite{Babai1977}.

\begin{lemma}\label{Ciso}
Let $X$ be a Cayley object of $G$ in some class ${\cal K}$ of combinatorial objects.  Then the following are equivalent:
\begin{enumerate}
\item\label{Ciso1} $X$ is a CI-object of $G$ in ${\cal K}$,
\item\label{Ciso2} whenever $\phi\in S_G$ such that $\phi^{-1}G_L\phi\le\Aut(X)$, $G_L$ and $\phi^{-1}G_L\phi$ are conjugate in $\Aut(X)$.
\end{enumerate}
\end{lemma}


For groups that are not CI-groups with respect to digraphs, Muzychuk \cite{Muzychuk1999} introduced the notion of a solving set.  The notion of a solving set is simply to give a set of permutations for which isomorphism between two Cayley digraphs of $G$ can be tested.  That is, two Cayley digraphs of $G$ are isomorphic if and only if they are isomorphic by an element of the solving set.  As the image of a Cayley digraph of $G$ under a group automorphism of $G$ is again a Cayley digraph of $G$, members of $\Aut(G)$ are typically in a solving set (although if an automorphism of $G$ is in the automorphism group of the Cayley digraph, they do not need to be formally in the solving set).  Additionally, by \cite[Lemma 4.2A]{DixonM1996}, $N_{S_G}(G_L) = \Aut(G)\cdot G_L$, and so $\Aut(G) = \Stab_{N_{S_G}(G_L)}(1_G)$.  That is, the elements which must always be tested to determine isomorphism normalize $G_L$ and fix $1_G$.  Later, in \cite{Dobson2014}, the notion of a CI-extension was developed.  The idea behind a CI-extension is to give an analogue of Babai's Lemma for solving sets.  Namely, a conjugation condition is given for a set $T$ of permutations in $S_G$, and a solving set for $G$ can then be obtained from $T$ and $\Aut(G)$.

There are then at least two ways in which one can generalize the Cayley isomorphism problem to the ``bi-Cayley Isomorphism Problem''.  Namely, one can determine the smallest set of possible isomorphisms that must be tested to determine isomorphism, or to test isomorphisms between Haar graphs using the elements of $N_{S_{\Z_2\times G}}(\widehat{G}_L)$ which fix $(0,1_G)$.  Unfortunately, and unlike in the Cayley isomorphism problem for digraphs, these two sets are not the same.  The original BCI problem, introduced in \cite{XuJSL2008}, calls a group $G$ a BCI-group if and only if two Haar graphs of $G$ are isomorphic by an element of $\bar{G}\widehat{A}$, which is the set of elements of $N_{S_{\Z_2\times G}}(\widehat{G}_L)$ that fix $(0,1_G)$.  However, the map $\tau$ given in Lemma \ref{normalizer} normalizes $\widehat{G}_L$, does not fix any points, and is not contained in $\widehat{G}_L$.  However, we note that the automorphism group of any Haar graph of an abelian group $A$ contains the map $(i,j)\mapsto(i+1,-j)$, which is a product of $\tau$ and the map $(i,j)\mapsto(i,-j)$, an automorphism of $\widehat{A}_L$.  This means that of the two possible ways one can generalize the Cayley isomorphism to the BCI problem, for abelian groups they are identical.

For nonabelian groups, though, the automorphism group of a Haar graph need not be vertex-transitive \cite{LuWX2006}, and so some nonabelian groups (for example, the dihedral group $D_p$ of order $2p$ where $p$ is prime and large enough) may be non-BCI-groups simply because there is some Haar graph whose automorphism group does not have an automorphism which switches the natural bipartition classes, while the normalizer of $\widehat{G}_L$ does.  To the authors, this does not really seem a good reason for a group to be non-BCI.  Thus we will use the other generalization of the CI-property and consider a smallest set of possible isomorphisms that must be tested to determine isomorphism.  We will call these groups ABCI-groups (for alternative bi-Cayley isomorphism), as we use the alternative generalization of the CI-property (the formal definition will be given below).  We first determine a list $L$ of maps contained in $N_{S_{\Z_2\times G}}(\widehat{G}_L)$ such that two Haar graphs of $G$ are isomorphic by an element of $N_{S_{\Z_2\times G}}(\widehat{G}_L)$ if and only if they are isomorphic by an element in $L$.  As this proof is entirely group theoretic and depends only upon $N_{S_{\Z_2\times G}}(\widehat{G}_L)$, we state and prove it in more generality as it requires no extra work.  We will, though, need an additional definition.

\begin{definition}
Let $G$ be a group, and $X$ a combinatorial object in some class ${\cal K}$ of combinatorial objects.  We say $X$ is a {\bf Haar object of $G$} if $\Aut(X)$, the automorphism group of $X$, contains $\widehat{G}_L$.
\end{definition}

\begin{lemma}\label{biCayley Iso}
Using the notation of Lemma \ref{normalizer}, two Haar objects of a group $G$ in some class ${\cal K}$ of combinatorial objects are isomorphic by an element of $N_{S_{\Z_2\times G}}({\widehat{G}_L})$ if and only if they are isomorphic by an element of $N_{S_{\Z_2\times G}}(\widehat{G}_L)$ of the form $\tau^i\widehat{\alpha}\bar{g}$, with $i\in\Z_2,\alpha\in\Aut(G)$, and $g\in G$.
\end{lemma}

\begin{proof}
Let $X$ and $Y$ be two isomorphic Haar objects of $G$ in some class ${\cal K}$ of combinatorial objects with $\phi\in N_{S_{\Z_2\times G}}(\widehat{G}_L)$ such that $\phi(X) = Y$.  By Lemma \ref{normalizer} we see that $\phi = \tau^i\widehat{\alpha}\widetilde{h}\bar{g'}$, with $i\in\Z_2$, $\alpha\in\Aut(G)$, and $h,g'\in G$.  Then $\widehat{h}_L^{-1}\in\Aut(X)$ and 
$$Y = \tau^i\widehat{\alpha}\widetilde{h}\bar{g'}\widehat{h}_L^{-1}(X) = \tau^i\widehat{\alpha}{\overline{g'h^{-1}}}(X)$$
and the result follows with $g = g'h^{-1}$ as the converse is trivial.
%
%
%
%
\end{proof}

This result shows that, theoretically, we only need to check elements of the form $\tau^i\widehat{\alpha}\bar{g}$ to check whether two Haar graphs are isomorphic by an element of $N_{S_{\Z_2\times G}}(\widehat{G}_L)$.  However, it is possible that a shorter list might suffice.  In \cite{MorrisS2024}, Joy Morris and Pablo Spiga showed that, with the exception of Haar graphs of abelian groups, there are only 20 groups $G$ for which there is no Haar graph with automorphism group $\widehat{G}_L$.  For Haar graphs $\Haar(G,S)$ with $G$ nonabelian and not on this list with $\Aut(\Haar(G,S)) = \widehat{G}_L$, the image of $\Haar(G,S)$ under any map of the form $\tau\widehat{\alpha}\bar{g}$ will be a Haar graph of $G$ different from $\Haar(G,S)$, and moreover, the images are also pairwise different (otherwise there will be an ``extra'' automorphism in $\Aut(\Haar(G,S))$).  For such groups, the shortest possible list of elements of the normalizers to check isomorphism will be those of the form $\tau\widehat{\alpha}\bar{g}$.   We can say more for Haar graphs of abelian groups.

\begin{lemma}\label{abelian Haar iso}
Using the notation of Lemma \ref{normalizer}, two Haar graphs of an abelian group $A$ are isomorphic by an element of $N_{S_{\Z_2\times A}}(\widehat{A}_L)$ of the form $\widehat{\alpha}\bar{a}$, with $\alpha\in\Aut(A)$, and $a\in A$.
\end{lemma}

\begin{proof}
If $T\subseteq A$ such that $\Haar(A,T)$ is isomorphic to $\Haar(A,S)$ by an element of $N_{S_{\Z_2\times G}}(\widehat{G}_L)$, then an isomorphism from $\Haar(A,S)$ to $\Haar(A,T)$ has the form $\tau^i\widehat{\alpha}\bar{a}$ for some $i\in\Z_2,\alpha\in\Aut(G)$, and $a\in A$ by Lemma \ref{biCayley Iso}.  Define $\iota: A\to A$ by $\iota(a) = -a$.  As $A$ is abelian, $\widehat{\iota}^{-i}\tau^{-i}\in\Aut(\Haar(A,T))$, and so $\widehat{\iota}^{-i}\tau^{-i}\tau^i\widehat{\alpha}\bar{a}(\Haar(A,S)) = \Haar(A,T)$ and $\widehat{\iota}^{-i}\tau^{-i}\tau^i\widehat{\alpha}\bar{a} = \widehat{-\alpha}\bar{a}$. 
\end{proof}

An infinite family of Haar graphs of nonabelian groups which are vertex-transitive but do not contain $\tau$ exists \cite[Corollary 10]{ConderEP2018}.  The element of the automorphism groups of these graphs that interchanges the bipartition sets also normalizes $\widehat{G}_L$ (this can be extracted from information given about their automorphism groups at the end of the proof of \cite[Theorem 8]{ConderEP2018}).  We have the obvious question of whether the previous result can always be generalized to vertex-transitive Haar graphs.  

\begin{problem}
Does there exist a vertex-transitive Haar graph $\Haar(G,S)$ for which there is no automorphism that interchanges the bipartition sets that normalizes $\widehat{G}_L$?
\end{problem}

When considering the CI problem, $\Aut(G)$ is the set of all elements of the normalizer of $G_L$ in $N_{S_G}(G_L)$ that fix $1_G$.  These are also all elements of the normalizer for which any two Cayley graphs of a group $G$ are isomorphic if they are isomorphic by an element of the normalizer.  Hence the {\it set} of elements $\{\tau^i\widehat{\alpha}\bar{g}:i\in\Z_2,\alpha\in A,g\in G\}$ for the alternative BCI problem plays the same role that $\Aut(G)$ plays for the CI problem.

\begin{definition}
Let $G$ be a group.  Using the notation of Lemma \ref{normalizer}, we define $\Iso(G)$ to be $\{\tau^i\widehat{\alpha}\bar{g}:i\in\Z_2,\alpha\in A,g\in G\}$.
\end{definition}

With the above lemma and discussion in mind, we make the following definitions.

\begin{definition}
Let $G$ be a group, and $S\subseteq G$.  We say that $\Haar(G,S)$ is an {\bf $\ABCI$-graph of $G$} if whenever $T\subseteq G$ then $\Haar(G,S)\cong\Haar(G,T)$ if and only if they are isomorphic by a map in $\Iso(G)$.    
\end{definition}

\begin{definition}
Let $G$ be a group.  We say that $G$ is an {\bf $\ABCI$-group} if whenever $S,T\subseteq G$ then $\Haar(G,S)$ and $\Haar(G,T)$ are isomorphic if and only if they are isomorphic by a map in $\Iso(G)$.
\end{definition}

Observe that if $\Haar(G,S)$ is a BCI-graph, then $\Haar(G,S)$ is an $\ABCI$-graph, and similarly, if $G$ is a BCI-group then $G$ is an $\ABCI$-group.  The converse, though, need not true.  We have the following problem:

\begin{problem}
Determine all $\ABCI$-groups.
\end{problem}

\section{ABCI-extensions}\label{ABCI extensions}

Most of the time, results on isomorphisms of Haar graphs focus on the BCI problem.  However, our main motivation for writing this paper is to see what the implications of isomorphisms of Cayley digraphs of $G$ are for isomorphisms of Haar graphs.  So we are interested in more than just the ABCI problem and thus need tools to solve the problem of finding isomorphisms between Haar graphs that are not ABCI-graphs.  The analogue for this problem in the CI problem are CI-extensions introduced in \cite{Dobson2014}, which give a systematic way to find solving sets for Cayley digraphs of a group $G$ introduced in \cite{Muzychuk1999}.  This very much works in the same way as the results in \cite{Dobson2014}, when one makes the obvious changes.  So the statements and proofs are very similar to those in \cite{Dobson2014}.  The proofs are included for completeness.  These types of results are mainly group theoretic, and as it will require little extra work, we work in the more general context of Haar objects.

\begin{definition}\label{solving set def}
Let $G$ be a finite group.  We say that $S\subseteq S_{\Z_2\times G}$ is a {\bf
solving set for a Haar object $X$} in a class of Haar objects
${\cal K}$ if the following hold:
\begin{enumerate}
\item for every $X'\in {\cal K}$ such that $X\cong X'$, there exists $s\in S$ such that $s(X) = X'$, 
\item if $s\in S$ and $s(0,1_G)\in B_0$ then $s(0,1_{G}) = (0,1_{G})$, and 
\item $\Iso(G)\subseteq S$. 
\end{enumerate}
\end{definition}

Condition (1) ensures that all Haar objects isomorphic to $X$ are isomorphic by a map in the solving set.  At this time, Condition (2) serves no obvious purpose.  We will point out explicitly where this condition is used in the proof of Lemma \ref{contool1} and thus the reason for its inclusion in the definition.  By Lemma \ref{biCayley Iso}, Condition (3) ensures that the solving set contains all elements of $N_{S_{\Z_2\times G}}(\widehat{G}_L)$ that must be checked for isomorphisms.

\begin{definition}\label{ABCI def}
Let $G$ be a finite group, and $X$ a Haar object of $G$ in a class ${\cal K}$ of combinatorial objects.  We define an {\bf ABCI-extension of $G$ with respect to $X$}, denoted $\ABCI(G,X)$, to be a set of permutations in $S_{\mathbb{Z}_{2} \times G}$ that satisfies the following properties:
\begin{enumerate}
    \item $\ABCI(G,X)$ contains the identity permutation in $S_{\Z_2\times G}$,
    \item if $t\in\ABCI(G,X)$ and $t(0,1_G)\in B_0$, then $t(0,1_G) = (0,1_G)$, and 
    \item whenever $\phi\in S_{\Z_2\times G}$ such that $\phi^{-1}\widehat{G}_L\phi\le\Aut(X)$, then there exists $v\in\Aut(X)$ such that $v^{-1}\phi^{-1}\widehat{G}_L\phi v = t^{-1}\widehat{G}_Lt$ for some $t\in \ABCI(G,X)$. 
\end{enumerate}  In the case where $X$ is a Haar graph (i.e. $X = \Haar(G,S)$), we will write $\ABCI(G,S)$.
\end{definition}

\begin{lemma}\label{existence}
Let $G$ be a finite group, and $X$ a Haar object of $G$ in some
class ${\cal K}$ of combinatorial objects.  Then an $\ABCI(G,X)$ exists.
\end{lemma}

\begin{proof}
To show existence, we only need show that there is a set of permutations $T$ in $S_{\Z_2\times G}$ that contains the identity such that whenever $\phi\in S_{\Z_2\times G}$ such that $\phi^{-1}\widehat{G}_L\phi\le\Aut(X)$ and $v\in\Aut(X)$, then $v^{-1}\phi^{-1}\widehat{G}_L\phi v = t^{-1}\widehat{G}_Lt$ for some $t\in T$ and $t(0,1_G) = (0,1_G)$ or $t(B_0) = B_1$.  This follows almost immediately.  If $\phi v(0,1_g)\not\in B_0$, let $t_{\phi v} = \phi v$.  If $\phi v(0,1_g)\in B_0$, then there exists $g\in G$ such that $\phi v(0,1_G) = (0,g)$.  Let $t_{\phi v} = \widehat{g}_L^{-1}\phi v$.  Then $t_{\phi v}(0,1_G) = \widehat{g}_L^{-1}\phi v(0,1_G) = \widehat{g}_L^{-1}(0,g) = (0,g^{-1}g) = (0,1_G)$, and existence is established with $T = \{1,t_{\phi v}:\phi^{-1}\widehat{G}_L\phi\le\Aut(X), v\in \Aut(X)\}$.
\end{proof}

%

Note that for $X$ a Haar object of $G$ in ${\cal K}$, an $\ABCI(G,X)$ is not unique as if $T$ is an ABCI-extension of $X$ with respect to $G$, then for $\alpha\in N_{S_{\Z_2\times G}}(\widehat{G}_L)$, $\{1,\widehat{\alpha} t:t\in T\}$ is also an ABCI-extension of $X$ with respect to $G$.  The following result shows the importance of $\ABCI(G,X)$, as if $\ABCI(G,X)$ is known, then the isomorphism problem is solved for $X$.

\begin{lemma}\label{contool1}
Let $G$ be a finite group, and $X$ a Haar object of $G$ in some class ${\cal K}$ of combinatorial objects.  Let $T\subseteq S_{\Z_2\times G}$ that contains the identity in $S_{\Z_2\times G}$ and if $t\in T$ such that $t(0,1_G)\in B_0$, then $t(0,1_G) = (0,1_G)$. The following are equivalent:
\begin{enumerate}
\item $S = \{\alpha t:\alpha\in \Iso(G), t\in T \}$ is a solving set for $X$,
\item $T$ is an $\ABCI(G,X)$.
\end{enumerate}
\end{lemma}

\begin{proof}
(1) implies (2).   Let $\phi\in S_{\Z_2\times G}$ such that
$\phi^{-1}\widehat{G}_L\phi\le\Aut(X)$.  Then $\phi(X)$ is a Haar
object of $G$ in ${\cal K}$  as $\Aut(\phi(X)) = \phi\Aut(X)\phi^{-1}\ge \widehat{G}_L$.  As $S$ is a solving set for $X$, $\phi(X) = s(X)$ for some $s\in S$, and $s = \alpha t$ for some $\alpha\in\Iso(G)$ and $t\in T$.  Thus $v =
\phi^{-1}s\in\Aut(X)$. Then
\begin{eqnarray*}
v^{-1}\phi^{-1}\widehat{G}_L\phi v & = & s^{-1}\phi \phi^{-1}\widehat{G}_L\phi\phi^{-1}s\\
  & = & s^{-1}\widehat{G}_Ls = t^{-1}\alpha^{-1}\widehat{G}_L\alpha t\\
  & = & t^{-1}\widehat{G}_Lt
\end{eqnarray*}
\noindent and $T$ is an $\ABCI(X,G)$.  

(2) implies (1).  Let $X$ and $X'$ be isomorphic Haar objects of $G$ in ${\cal K}$.  Then there exists $\phi\in S_{\Z_2\times G}$ such that $\phi(X) = X'$.  As $\widehat{G}_L\le\Aut(X')$, $\phi^{-1}\widehat{G}_L\phi\le\Aut(X)$. As $T$ is an $\ABCI(G,X)$, there exists $t\in T$ and $v\in\Aut(X)$ such that $v^{-1}\phi^{-1}\widehat{G}_L\phi v = t^{-1}\widehat{G}_Lt$.  Hence $tv^{-1}\phi^{-1}\widehat{G}_L\phi v t^{-1} = \widehat{G}_L$. Then $\phi vt^{-1}$ normalizes $\widehat{G}_L$, and so by Lemma \ref{normalizer}, we have $\phi vt^{-1} = \tau^i\widehat{a}\widetilde{g}\bar{h}$ for $i\in\Z_2$, $a\in \Aut(G)$, and $g,h\in G$.  Set $A = \Aut(G)$, and $\gamma = \tau^i\widehat{a}\widetilde{g}\bar{h}$.  Then $\phi v = \gamma t$ and $\phi v(X) = \phi(X) = X'$.  Recall that $\widehat{A}\cdot(\widetilde{G}\bar{G})$ contains $\widehat{G}_L$. As $\widehat{a}\widetilde{g}(B_0) = B_0$, there exists $k\in G$ such that $\widehat{k}_L\widehat{a}\widetilde{g}(0,1_G) = (0,1_G)$.  As $\widehat{k}_L\tau^i = \tau^i\widehat{k}_L$, we see that $\widehat{k}_L\gamma = \widehat{k}_L\tau^i\widehat{a}\widetilde{g}\bar{h} = \tau^i\widehat{k}_L\widehat{a}\widetilde{g}\bar{h}$ and $\widehat{k}_L\widehat{a}\widetilde{g}\bar{h}(0,1_G) = \widehat{k}_L\widehat{a}\widetilde{g}(0,1_G) = (0,1_G)$.  By Lemma \ref{normalizer}, we may rewrite $\tau^i\widehat{k}_L\widehat{a}\widetilde{g}\bar{h}$ as an element $\tau^i\widehat{b}\widetilde{\ell}\bar{m}$, where $b\in A$ and $\ell,m\in G$.  We see that $\widehat{b}\widetilde{\ell}\bar{m}(0,1_G) = (0,1_G)$, and as $\widehat{b}$ and $\bar{m}$ both fix $(0,1_G)$, we deduce that $\ell = 1$.  Thus the inclusion of the condition that if $s\in S$ and $s(0,1_G)\in B_0$ then $s(0,1_G) = (0,1_G)$ in Definition \ref{solving set def} allows us to eliminate any map of the form $\widetilde{\ell}$ from the maps to be considered for isomorphism.  This is the reason for the inclusion of (2) in Definition \ref{solving set def}.  Also, $\tau^i\widehat{b}\bar{m}\in\Iso(G)$ and maps $(0,1_G)$ to $(0,1_G)$ if $i = 0$ or, if $i = 1$, maps $B_0$ to $B_1$.  Finally, $\widehat{k}_L\gamma t(X) = \widehat{k}_L(X') = X'$ and $\widehat{k}_L\gamma = \tau^i\widehat{b}\bar{m}\in\Iso(G)$.  The result follows with $\alpha = \widehat{k}_L\gamma$.  
\end{proof}

In Definition \ref{ABCI def} of an $\ABCI(G,X)$, we can now see that Condition (1) ensures that the solving set contains $\Iso(G)$.  Condition (2) of Definition \ref{ABCI def} ensures that Condition (2) of Lemma \ref{contool1} holds.  So conditions (1) and (2) are in some sense ``technical'' conditions.  Condition (3) of Definition \ref{ABCI def} can be thought of as an analogue of Lemma \ref{Ciso} (2).  An analogue of the next result exists for the BCI problem (see \cite[Corollary 2.2]{KoikeK2014} and \cite[Lemma 2.2]{KoikeK2019}). 

\begin{corollary}
Let $G$ be a group, and $X$ a Haar object in a class ${\cal K}$ of combinatorial objects.  The following are equivalent:
\begin{enumerate}
\item $X$ is an $\ABCI$-object of $G$,
\item whenever $\phi\in S_{\Z_2\times G}$ such that $\phi^{-1}\widehat{G}_L\phi\le\Aut(X)$, then $\phi^{-1}\widehat{G}_L\phi$ and $\widehat{G}_L$ are conjugate in $\Aut(X)$.
\end{enumerate}
\end{corollary}

\begin{proof}
Observe that $X$ is an $\ABCI$-object of $G$ if and only if $\ABCI(G,X) = \{1\}$.  The result follows by Lemma \ref{contool1}.
\end{proof}

The following result shows that if a solving set for $X$ has been found, then some $\ABCI(G,X)$ has also been found.

\begin{lemma}\label{CIX}
Let $G$ be a group, $X$ a Haar object of $G$, and $S$ a solving set for $X$.  Define an equivalence relation $\equiv$ on $S$ by $s_1\equiv s_2$ if and only if $s_1 = \nu s_2$ for some $\nu\in \Iso(G)$.  Let $T$ be a set consisting of one representative of each equivalence class of $\equiv$, with $1$ chosen from the equivalence class $\Iso(G)$.  Then $T$ is an $\ABCI(G,X)$.
\end{lemma}

\begin{proof} 
It is straightforward to verify that $\equiv$ is an equivalence relation and that $\Iso(G,X)$ is an equivalence class of $\equiv$.  By construction, $T$ contains the identity in $S_{\Z_2\times G}$, so Definition \ref{ABCI def} (1) holds.  Definition \ref{ABCI def} (2) holds as $T\subseteq S$ and $S$ is a solving set.  For Definition \ref{ABCI def} (3), choose $T$ as is given in the statement.   Let $X'$ be a Haar object of $G$ isomorphic to $X$ with $\phi:X\to X'$ an isomorphism.  Then $\phi^{-1}\widehat{G}_L\phi\le\Aut(X)$.  Also, as $S$ is a solving set for $X$, there exists $s\in S$ such that $s(X) = X'$ so that $v = \phi^{-1}s\in\Aut(X)$.  Let $t\in T$ such that $t\equiv s$ so that $\nu t = s$ for some $\nu\in \Iso(G)$.  Then 
$$v^{-1}\phi^{-1}\widehat{G}_{L}\phi v = s^{-1}\widehat{G}_{L} s = t^{-1}\nu^{-1}\widehat{G}_{L}\nu t = t^{-1}\widehat{G}_{L} t.$$
Thus $T$ is an $\ABCI(G,X)$.
\end{proof}

We now more or less repeat the above results in the more general context of solving sets for classes of combinatorial objects.  The proofs are basically the same as the proofs for a solving set for a Haar object of $G$ as above, and so are not repeated.  

\begin{definition}
Let $G$ be a finite group.  We say that $S\subseteq S_{\Z_2\times G}$ is a {\bf solving set for a class ${\cal
K}$ of Haar objects of $G$} if the following hold:
\begin{enumerate}
\item whenever $X,X'\in {\cal K}$ are Haar objects of $G$ and $X\cong X'$, then there exists $s\in S$ such that $s(X) = X'$, 
\item if $s\in S$ and $s(0,1_G)\in B_0$, then $s(0,1_{G}) = (0,1_{G})$, and 
\item $\Iso(G)\subseteq S$.
\end{enumerate}
\end{definition}

\begin{definition}
Let $G$ be a finite group, and ${\cal K}$ be a class of combinatorial objects.  We define an {\bf ABCI-extension of $G$ with respect to ${\cal K}$}, denoted by $\ABCI(G,{\cal K})$, to be a set of permutations in $S_{\Z_2\times G}$ that satisfies the following properties:
\begin{enumerate}
    \item $\ABCI(G,{\cal K})$ contains the identity permutation in $S_{\Z_2\times G}$,
    \item if $t\in\ABCI(G,{\cal K})$ and $t(0,1_G)\in B_0$, then $t(0,1_G) = (0,1_G)$, and 
    \item whenever $X\in{\cal K}$ is a Haar object of $G$ and $\phi\in S_{\Z_2\times G}$ such that $\phi^{-1}\widehat{G}_L\phi\le\Aut(X)$, then there exists $t\in \ABCI(G,{\cal K})$ and $v\in\Aut(X)$ such that $v^{-1}\phi^{-1}\widehat{G}_L\phi v = t^{-1}\widehat{G}_Lt$.
\end{enumerate}
\end{definition}

\begin{lemma}\label{contool2}
Let $G$ be a finite group, and ${\cal K}$ a class of combinatorial objects.  Let $T\subseteq S_{\Z_2\times G}$ that contains the identity and if $t\in T$ such that $t(B_0) = B_0$, then $t(0,1_G) = (0,1_G)$.  The following are equivalent:
\begin{enumerate}
\item $S = \{\alpha t:\alpha\in\Iso(G), t\in T\}$ is a solving set for $G$ in ${\cal K}$,
\item $T$ is an $\ABCI(G,{\cal K})$.
\end{enumerate}
\end{lemma}

\begin{corollary}\label{contool3}
Let $G$ be a group, and ${\cal K}$ a class of combinatorial objects.  The following are equivalent:
\begin{enumerate}
\item $G$ is an $\ABCI$-group with respect to ${\cal K}$,
\item whenever $X$ is a Haar object of $X$ in ${\cal K}$ and $\phi\in S_{\Z_2\times G}$ such that $\phi^{-1}\widehat{G}_L\phi\le\Aut(X)$, then $\phi^{-1}\widehat{G}_L\phi$ and $\widehat{G}_L$ are conjugate in $\Aut(X)$.
\end{enumerate}
\end{corollary}


\begin{lemma}\label{CIK}
Let $G$ be a group, ${\cal K}$ a class of combinatorial objects, and $S$ a solving set for $\widehat{G}_L$ in ${\cal K}$.  Define an equivalence relation $\equiv$ on $S$ by $s_1\equiv s_2$ if and only if $s_1 = \nu s_2$ for some $\nu\in \Iso(G)$.  Let $T$ be a set consisting of one representative of each equivalence class of $\equiv$, with $1$ chosen from the equivalence class $\Iso(G)$.  Then $T$ is an $\ABCI(G,X)$.
\end{lemma}

We now become even more general, and consider the isomorphism problem for Haar objects of a group $G$.  Again, the proofs are basically the ones we have already seen and so we do not repeat them.

\begin{definition}
Let $G$ be a finite group.  A set $S$ is a {\bf solving set for $\widehat{G}_L$} if the following hold:
\begin{enumerate}
\item whenever $X,X'$ are isomorphic Haar objects of $G$ in any class ${\cal K}$ of combinatorial objects, then there exists $s\in S$ such that $s(X) = X'$, 
\item if $s\in S$ and $s(0,1_G)\in B_0$, then $s(0,1_{G}) = (0,1_{G})$, and 
\item $\Iso(G)\subseteq S$.
\end{enumerate}
\end{definition}

\begin{definition}
Let $G$ be a finite group.  We define an {\bf ABCI-extension of $G$}, denoted by $\ABCI(G)$, to be a set of permutations in $S_{\Z_2\times G}$ that satisfies the following properties:
\begin{enumerate}
    \item $\ABCI(G)$ contains the identity permutation in $S_{\Z_2\times G}$,
    \item if $t\in\ABCI(G)$ and $t(0,1_G)\in B_0$, then $t(0,1_G) = (0,1_G)$, and  
    \item whenever $X\in{\cal K}$ is a Haar object of $G$ in some class ${\cal K}$ of combinatorial objects, and $\phi\in S_{\Z_2\times G}$ such that $\phi^{-1}\widehat{G}_L\phi\le\Aut(X)$, then there exists $t\in \ABCI(G)$ and $v\in\la \widehat{G}_L,\phi^{-1}\widehat{G}_L\phi\ra$ such that
$v^{-1}\phi^{-1}\widehat{G}_L\phi v = t^{-1}\widehat{G}_Lt$.
\end{enumerate}   
\end{definition}

Repeated application of Lemma \ref{existence} for every combinatorial object $X$ in every class ${\cal K}$ of combinatorial objects shows that $\ABCI(G)$ exists.  The proofs of the following results are also straightforward.

\begin{lemma}
Let $G$ be a finite group.  Let $T\subseteq S_{\Z_2\times G}$ that contains the identity and if $t\in T$ such that $t(B_0) = B_0$, then $t(0,1_G) = (0,1_G)$.  The following are equivalent:
\begin{enumerate}
\item $S = \{\alpha t:\alpha\in\Iso(G), t\in T\}$ is a solving set for $\widehat{G}_L$,
\item $T$ is an $\ABCI(G)$.
\end{enumerate}
\end{lemma}

\begin{corollary}
Let $G$ be a group.  The following are equivalent:
\begin{enumerate}
\item $G$ is an $\ABCI$-group,
\item whenever $X$ is a Haar object of $G$ and $\phi\in S_{\Z_2\times G}$ such that $\phi^{-1}\widehat{G}_L\phi\le\Aut(X)$, then $\phi^{-1}\widehat{G}_L\phi$ and $\widehat{G}_L$ are conjugate in $\Aut(X)$.
\end{enumerate}
\end{corollary}

\begin{lemma}\label{CIG}
Let $G$ be a group, and $S$ a solving set for $\widehat{G}_L$.  Define an equivalence relation $\equiv$ on $S$ by $s_1\equiv s_2$ if and only if $s_1 = \nu s_2$ for some $\nu\in \Iso(G)$.  Let $T$ be a set consisting of one representative of each equivalence class of $\equiv$, with $1$ chosen from the equivalence class $\Iso(G)$.  Then $T$ is an $\ABCI(G,X)$.
\end{lemma}

\section{Isomorphism criteria for Haar graphs}\label{three families}

We begin with the necessary definitions to understand the statement of \cite[Corollary 2.16]{Dobson2022}, which characterizes the automorphism groups of Haar graphs of abelian groups.

\begin{definition}
Let $\Gamma_1$ and $\Gamma_2$ be graphs. The \textbf{wreath product of $\Gamma_1$ and $\Gamma_2$}, denoted $\Gamma_1\wr\Gamma_2$,  is the graph with vertex set $V(\Gamma_1)\times V(\Gamma_2)$ and edge set
$$\{(u,v)(u,v'):u\in V(\Gamma_1){\rm\ and\ } vv'\in E(\Gamma_2)\}\cup\{(u,v)(u',v'): uu'\in E(\Gamma_1){\rm\ and\ }v,v'\in V(\Gamma_2)\}.$$
\end{definition}

\begin{definition}
Let $G\le S_X$ and $H\le S_Y$.  Define the {\bf wreath product of $G$ and $H$}, denoted $G\wr H$, to be the set of all permutations of $X\times Y$ of the form $(x,y)\mapsto (g(x),h_x(y))$, where $g\in G$ and each $h_x\in H$.
\end{definition}

\begin{definition}
A {\bf permutation representation} of a group $G$ is a homomorphism $\phi\colon G\to S_n$ for some $n$.
\end{definition}

If $G$ acts on a set $X$, then there is a homomorphism $\phi:G\to S_X$.  So an action of $G$ induces a permutation representation of $G$.

\begin{definition}
Let $G$ be a group, and $X$ and $Y$ sets.  Let $\beta\colon G\to S_X$ and $\delta\colon G\to S_Y$ be permutation representations of $G$.  We say $\beta$ and $\delta$ are {\bf equivalent permutation representations} of $G$ if there exists a bijection $\lambda\colon X\to Y$ such that $\lambda(\beta(g)(x)) = \delta(g)(\lambda(x))$ for all $x\in X$ and $g\in G$.  In this case, we will say that $\beta(G)$ and $\delta(G)$ are {\bf permutation equivalent}. 
\end{definition}

The next result \cite[Corollary 2.16]{Dobson2022} characterizes automorphism groups of Haar graphs of abelian groups, and will be one of the primary tools in the proof of the main results.   

\begin{theorem}\label{automorphism main result} 
Let $A$ be an abelian group, $S\subseteq A$, and $\Gamma = \Haar(A,S)$.  Then one of the following is true: 
\begin{enumerate}
\item\label{first} $\Gamma$ is disconnected, and there is $a\in A$ and $H < A$ such that $$\Gamma = \bar{a}^{-1}(\Haar(A,a + S))$$ and $\Aut(\Gamma) \cong \bar{a}^{-1}(S_{A/H}\wr\Aut(\Haar(H,a + S)))\bar{a}$, 
\item\label{second} $\Gamma$ is connected, and the action of the setwise stabilizer of each $B\in{\cal B}$ in $\Aut(\Gamma)$ is unfaithful on $B_1$.  There exists a subgroup $1 < H\le A$ such that $\Gamma\cong \Haar(A/H,S)\wr\bar{K}_\beta$ where $\beta = \vert H\vert$ and $S$ is interpreted as a set of cosets of $H$ in $A$, and $\Aut(\Gamma)\cong\Aut(\Haar(A/H,S))\wr S_\beta$.  Additionally, denoting the natural bipartition of $\Haar(A/H,S)$ by ${\cal D}$, the action of the setwise stabilizer of each $D\in{\cal D}$ on $D\in{\cal D}$ is faithful, 
\item\label{third} $\Aut(\Gamma) = \bar{a}^{-1}\Z_2\ltimes\Aut(\Cay(A,a + S))\bar{a}$ for some $a\in A$, or 
\item\label{fourth} the action of the setwise stabilizer $F$ of each $B\in {\cal B}$ in $Aut(\Gamma)$ on $B_1$ is faithful on $B_1$ but the induced actions of $F$ on $B_0$ and $B_1$ are not equivalent permutation representations of $F$. 
\end{enumerate}
\end{theorem}

A word about the isomorphisms in (1) and (2) in the preceding result.  They are necessary here only because the vertex set of a wreath product is a direct product, while $A$ need not be the same direct product.  To obtain equality, one only need replace the fibers $(g,H)$ with the left coset $gH$.

The rest of this section more or less consists of showing that the isomorphism problem when conditions (1), (2), and (3) of the previous result hold reduces to the isomorphism problem for Cayley digraphs of $A$, or the isomorphism problem for Haar graphs of quotients or subgroups of $A$.  We begin with the isomorphism problem for Haar graphs corresponding to Theorem \ref{automorphism main result} (3). 

\begin{lemma}\label{disconnected a is 0}
Let $A$ be an abelian group of odd order, and $S\subseteq A$ such that $S = -S$ and $\Haar(A,S)$ is disconnected.  Then $S\subseteq H$ for some $H\le A$.
\end{lemma}

\begin{proof}
As $\Haar(A,S)$ is disconnected, by \cite[Lemma 1.8]{Dobson2022} $S\subseteq a + H$ for some $H\le A$ that we choose to be minimal, and $a\in A$.  As $S = -S$, $-a+H\subseteq a + H$, and so $2a + H\subseteq H$.  Then $2a\in H$, and $a + H$ has order dividing $2$ in $A/H$.  As $\vert A\vert$ is odd, $a + H = H$.  Thus $S\subseteq H$.
\end{proof}

\begin{lemma}\label{connected preserve bipartition}
Let $G$ be a group, and $S\subseteq G$.  If there exists $\phi\in S_{\Z_2\times G}$ such that $\phi^{-1}\widehat{G}_L\phi\le\Aut(\Haar(G,S))$ and does not have ${\cal B}$ as its orbits, then $\Haar(G,S)$ is disconnected.
\end{lemma}

\begin{proof}
As ${\cal B}$ is a bipartition of a Haar graph of $G$, $\phi^{-1}({\cal  B})$ is also a bipartition of $\Haar(G,S)$.  As the orbits of $\phi^{-1}\widehat{G}_L\phi$ are not ${\cal B}$, $\phi^{-1}({\cal B})\not = {\cal B}$.  Let $i\in\Z_2$.  Then $\Haar(G,S)$ has no edges in $B_i\in{\cal B}$ or $\phi^{-1}(B_i)$.  Thus the only edges in $\Haar(G,S)$ are between $\phi^{-1}(B_i)$ and $\phi^{-1}(B_{i+1})$.  As ${\cal B}\not = \phi^{-1}({\cal B})$, $X_0 = B_0\cap\phi^{-1}(B_0)\not = \emptyset$ and $X_1 = B_0\cap\phi^{-1}(B_1)\not = \emptyset$.  Similarly, $X_2 = B_1\cap\phi^{-1}(B_1)\not = \emptyset$ and $X_3 = B_1\cap\phi^{-1}(B_0)\not = \emptyset$.  Then the vertices in $X_0$ can only be adjacent to the vertices in $X_2$ and vice versa, and the vertices in $X_1$ can only be adjacent to the vertices in $X_3$, and vice versa.  As $X_0\cup X_1\cup X_2\cup X_3 = \Z_2\times G$, the only edges in $\Haar(G,S)$ are contained in $X_0\cup X_2$ and $X_1\cup X_3$.  Hence $\Haar(G,S)$ is disconnected.
\end{proof}

We now need the formal definitions of a Cayley extension and solving set of a Cayley digraph $\Cay(G,S)$.

\begin{definition}
Let $G$ be a finite group.  We say that $S\subseteq S_G$ is a {\bf
solving set for a Cayley object $X$} in a class of Cayley objects
${\cal K}$ if for every $X'\in {\cal K}$ such that $X\cong X'$,
there exists $s\in S$ such that $s(X) = X'$, $s(1_G) = 1_G$ for every $s\in S$, and $\Aut(G)\le S$. 
\end{definition}

\begin{definition}
Let $X$ be a Cayley object of $G$ in ${\cal K}$.  We define a {\bf CI-extension of $G$ with respect to $X$}, denoted by $\CI(G,X)$, to be a set of permutations in $S_G$ that each fix $1_G$ and whenever $\delta\in S_G$ such that $\delta^{-1}G_L\delta\le\Aut(X)$, then there exists $v\in\Aut(X)$ such that $v^{-1}\delta^{-1}G_L\delta v = t^{-1}G_Lt$ for some $t\in \CI(G,X)$.
\end{definition}

In \cite[Lemma 20]{Dobson2014} it was shown that if a CI-extension of $G$ with respect to $X$ has been found, then so has a solving set.

\begin{lemma}\label{Reduction in Case 3}
Let $G$ be a group and $S\subseteq G$ such that $\Aut(\Haar(G,S)) = \bar{a}\Z_2\ltimes \Aut(\Cay(G,S))\bar{a}^{-1} = \Z_2\ltimes\Aut(\Cay(G,Sa))$ for some $a\in G$.  Let $T$ be a $\CI(G,Sa)$.  Then $\widehat{T}\bar{a}$ is an $\ABCI(G,S)$.
\end{lemma}

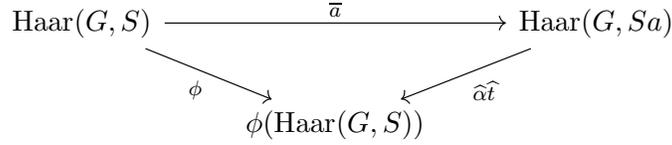
\begin{figure}

\begin{center}
\begin{tikzcd}
    \Haar(G,S) \arrow[dr, swap, "\phi"] \arrow[rr, "\overline{a}"] & & \Haar(G,Sa) \arrow[dl, "\widehat{\alpha} \widehat{t}"]\\
    & \phi(\Haar(G,S))  
\end{tikzcd}
\end{center}
\caption{The maps in Lemma \ref{Reduction in Case 3}}
\label{maps 1}
\end{figure}

\begin{proof}
The maps used in this proof, along with corresponding graphs, are shown in Figure \ref{maps 1}.  Let $\Gamma = \Haar(G,S)$.  First observe that by hypothesis, $\Gamma$ is vertex-transitive. Let $\phi\in S_{\Z_2\times G}$ such that $\phi^{-1}\widehat{G}_L\phi\le\Aut(\Gamma)$.  As $\Aut(\Gamma) = \bar{a}\Z_2\ltimes\Aut(\Cay(G,S))\bar{a}^{-1}$, we see that $\Gamma$ is connected, as a disconnected graph has automorphism group a wreath product.  As $\Gamma$ is connected, ${\cal B}$ is a block system of $\Aut(\Gamma)$, and $\phi({\cal B}) = {\cal B}$ by Lemma \ref{connected preserve bipartition}.  Then $\bar{a}\phi^{-1}\widehat{G}_L\phi \bar{a}^{-1}\le\Aut(\Haar(G,Sa)) = \Z_2\ltimes\Aut(\Cay(G,Sa))$. Note that ${\cal B}$ is also a block system of $\Z_2\ltimes\Aut(\Cay(G,Sa))$.  Also, if $\gamma\in\fix_{\Z_2\ltimes\Aut(\Cay(G,Sa))}({\cal  B})$, then $\gamma = \widehat{\delta}$ for some $\delta\in S_G$.  As there is $\kappa\in\Aut(\Gamma)$ with $\kappa(B_1) = B_0$, replacing $\phi$ with $\phi\kappa$, we may assume without loss of generality that $\phi$ fixes each $B\in{\cal B}$.  This implies that if $\omega\in S_{\Z_2\times G}$ such that $\omega^{-1}\widehat{G}_L\omega\le\fix_{\Aut(\Haar(G,Sa))}({\cal B})$ and there is $v\in\Aut(\Cay(G,Sa))$ and $t\in \CI(G,Sa)$ such that $v^{-1}(\omega^{-1})^{B_0}G_L\omega^{B_0} v = t^{-1}G_Lt$, then $\widehat{v}^{-1}\omega^{-1}\widehat{G}_L\omega\widehat{v} = \widehat{t}^{-1}\widehat{G}_L\widehat{t}$.  We then see that there is $\widehat{v}\in\Aut(\Haar(G,Sa))$ and $\widehat{t}\in \widehat{\CI(G,Sa)}$ such that $\widehat{v}^{-1}\bar{a}\phi^{-1}\widehat{G}_L\phi \bar{a}^{-1}\widehat{v} = \widehat{t}^{-1}\widehat{G}_L\widehat{t}$.  Then $$\bar{a}^{-1}\widehat{v}^{-1}\bar{a}\phi^{-1}\widehat{G}_L\phi \bar{a}^{-1}\widehat{v}\bar{a} = \bar{a}^{-1}\widehat{t}^{-1}\widehat{G}_L\widehat{t}\bar{a}.$$ As $\widehat{v}\in\Aut(\Haar(G,Sa))$ if and only if $\bar{a}^{-1}\widehat{v}\bar{a}\in\Aut(\Haar(G,S))$ we have that an $\ABCI(G,S) = \widehat{\CI(G,Sa)}\bar{a}$.  For Figure \ref{maps 1}, this means that there is $\alpha\in\Aut(G)$ and $t\in T$ such that $\widehat{\alpha}\widehat{t}$ is an isomorphism from $\Haar(G,Sa)$ to $\phi(\Haar(G,S))$.  This means that $\Gamma$ and $\phi(\Gamma)$ are isomorphic by $\widehat{\alpha}\widehat{t}\bar{a}$, and as $\phi$ was arbitrary, $\widehat{T}\bar{a}$ is an $\ABCI(G,S)$. 
\end{proof}

\begin{corollary}\label{Stable CI}
Let $G$ be a group and $S\subseteq G$ such that $$\Aut(\Haar(G,S)) = \bar{a}^{-1}\Z_2\ltimes\Aut(\Cay(G,Sa))\bar{a}$$ for some $a\in A$.
If $\Cay(G,Sa)$ is a CI-digraph of $G$, then $\Haar(G,S)$ is an $\ABCI$-graph of $G$.
\end{corollary}

\begin{proof}
By Lemma \ref{Reduction in Case 3}, $\widehat{T}\bar{a}$ is an $\ABCI(G,S)$, where $T$ is an $\CI(G,Sa)$.  As  $\Cay(G,Sa)$ is a CI-digraph of $G$, $T\le\Aut(G)$ and so $\widehat{T}\subset \Iso(G)$.  Then $\widehat{T}\bar{a}\subset \Iso(G)$, and $\Haar(G,S)$ is an $\ABCI$-graph of $G$.
\end{proof}

This completes the isomorphism problem for Haar graphs corresponding to Theorem \ref{automorphism main result} (3).  We now turn to the isomorphism problem for Haar graphs corresponding to Theorem \ref{automorphism main result} (2).  We begin with the necessary terms and notation.

\begin{definition}
Let $\Gamma$ be a vertex-transitive digraph whose automorphism group contains a transitive subgroup $G$ with a block system ${\cal B}$.  Define the {\bf block quotient digraph of $\Gamma$ with respect to ${\cal B}$}, denoted $\Gamma/{\cal B}$, to be the digraph with vertex set ${\cal B}$ and arc set $\{(B,B'):B,B'\in{\cal B}, B\neq B',{\rm\ and\ }(u,v)\in A(\Gamma){\rm\ for\ some\ }u\in B{\rm\ and\ } v \in B'\}$.
\end{definition}

See \cite[Section 3.6]{Book} for more information about block quotient digraphs as well as examples.

\begin{definition}
Let $A$ be an abelian group and $S\subseteq A$ such that $\Haar(A,S)$ is connected. We denote the set-wise stabilizer of the bipartition ${\cal B}$ by $F_{A,S}$.
\end{definition}

Recall that if $\Haar(A,S)$ is connected, then ${\cal B}$ is a block system of $\Aut(\Haar(A,S))$.  Hence $F_{A,S} = \fix_{\Aut(\Haar(A,S))}({\cal B})$.

\begin{definition}
Let $G$ be a group, $C\tl G$, and $\phi:G/C\to G/C$ be a function.  A function $\phi_G:G\to G$ is an {\bf extension of $\phi$ to $G$} if $\phi_G$ maps left cosets of $C$ to left cosets of $C$ and $\phi_G(gC) = \phi(gC)$ for every $g\in G$.
\end{definition}

\begin{definition}
Let $G\le S_X$ and $H\le S_Y$.  Let $C_x = \{x\}\times Y$, $x\in X$, and ${\cal C} = \{C_x:x\in X\}$.  Then ${\cal C}$ is a block system of $G\wr H$.  It is called the {\bf lexi-partition of $G\wr H$ corresponding to $Y$}.  
\end{definition}

\begin{definition}
A graph $\Gamma$ is {\bf reducible} if there exist distinct vertices $u,v\in V(\Gamma)$ such that the neighbors of $u$ and $v$ in $\Gamma$ are the same.  A graph that is not reducible is {\bf irreducible}.
\end{definition}

The `reducible' terminology is due to Sabidussi \cite{Sabidussi1964}.  The same idea has been independently considered by many other authors, including Kotlov and Lov\'asz \cite{KotlovL1996}, who called the vertices $u,v\in V(\Gamma)$ with the same neighbors in $\Gamma$ {\bf twins}. Consequently, an irreducible graph is often called {\bf twin-free}.  The reducible/irreducible terminology is often used when using the wreath product (as we will use in this section), and the twin/twin-free terminology is often used when considering unstable graphs (which we will use in Section \ref{applications}).  Note that the graphs satisfying Theorem \ref{automorphism main result} (\ref{second}) are always reducible.


Isomorphisms of Haar graphs that are wreath products are straightforward once one is aware of certain examples.  Let $p$ be an odd prime.  In \cite[Example 6.4]{DobsonM2009} an example is given of a Cayley digraph of $\Z_p\times\Z_{p^2}$ that is a wreath product with the property that its automorphism group has a unique block system ${\cal C}$ with blocks of size $p$, and it has two different subgroups isomorphic to $\Z_p\times\Z_{p^2}$ whose quotients by ${\cal C}$ are not isomorphic.  One quotient is $\Z_{p^2}$ and the other is $\Z_p\times\Z_p$ (most of the stated information must be extracted from the proof).  For us, this means that it is possible to have two Cayley digraphs of different groups whose wreath product with another Cayley digraph yields two Cayley digraphs of the same group $G$.  That is, to establish isomorphism between the two wreath products, we must consider all the possible quotient groups of $G$ that are contained in the automorphism group of the quotient graph, not just the one with quotient $G_L/{\cal C}$.  The same sort of problem also arises for Haar graphs (simply consider the Haar graphs corresponding to the Cayley graphs mentioned earlier in this paragraph).

We now also restrict our attention to Haar graphs of abelian groups.  There are two reasons for this.  First, our main applications in the next section apply only to Haar graphs of abelian groups.  Second, and more importantly, we will need to consider block quotient digraphs, and the block quotient digraph of a Haar graph of a nonabelian group $G$ need not be a Haar graph if the quotient is the set of left cosets of a non-normal subgroup of $G$, but may be a more general graph called a bicoset graph (see \cite{DuX2000}) which we will not consider here.


We will need a new idea for our next result, which will also be used later to determine the isomorphisms between disconnected Haar graphs of abelian groups where the sets of components are not the same.


\begin{definition}
A group $G\le S_n$ is {\bf $G$-semitransitive} if $G$ has exactly two orbits.
\end{definition}

The following result is a special case of \cite[Lemma 2.4]{DuX2000} together with some facts from its proof.  The result \cite[Lemma 2.4]{DuX2000} is the analogue of Sabidussi's Theorem \cite[Lemma 4]{Sabidussi1958} or \cite[Theorem 1.2.20]{Book} that a vertex-transitive digraph is isomorphic to a double coset digraph for bipartite graphs.   What we state here is the special case when $G$ is semiregular.

\begin{lemma}\label{Sabidussi biCayley}
Suppose $\Gamma$ is a $G$-semitransitive graph with $U(\Gamma)$ and $W(\Gamma)$ the orbits of $G$ such that $V(\Gamma) = U(\Gamma) \cup W(\Gamma)$ is a bipartition of $\Gamma$. Take $u\in U (\Gamma)$ and $w\in W(\Gamma)$. Set $S = \{g\in G: g(w)\in N_{\Gamma}(u)\}$. If $G$ is semiregular, then $\Gamma \cong \Haar(G,S)$ by the map that sends $g(u)$ to $gH_0$ and $g(w)$ to $gH_1$.
\end{lemma}

Let $A$ be an abelian group and $S\subseteq A$.  Suppose that $\Haar(A,S)$ is also a $G$-semitransitive graph with $G\le\Aut(\Haar(A,S))$ abelian and semiregular with two orbits which are also $B_0$ and $B_1$.  That is, the bipartition ${\cal B}$ of $\Haar(A,S)$ is also the set of orbits of $G$.  We wish to apply Lemma \ref{Sabidussi biCayley} and so choose $u = (0,0)$ and $w = (1,0)$.  
Then $\Haar(A,S)$ is isomorphic to $\Haar(G,T)$, where $T = \{g\in G:g(1,0)\in N_{\Haar(A,S)}(0,0)\}$ by the map, call it $\phi$, that sends $g(0,0)$ to $(0,g)$ and $g(1,0)$ to $(1,g)$.  Of course, all $\phi$ really does is relabel $\Z_2\times A$ with elements of $\Z_2\times G$ as given by $\phi$.  We call $\phi$ an {\bf $A$ to $G$ relabeling of $\Haar(A,S)$}.  Also, $\phi(\Haar(A,S)) = \Haar(G,T)$ is still isomorphic to a Haar graph of $A$, but it no longer {\it is} a Haar graph of $A$.

Note that in the following lemma, we found it easier to directly construct a solving set, rather than finding an ABCI-extension.

\begin{lemma}\label{action unfaithful 1}
Let $A$ be an abelian group, $S\subseteq A$ such that $\Haar(A,S)$ is connected, and there exists $C < A$, chosen to be maximal, such that $\Haar(A,S)\cong \Haar(D,U)\wr \bar{K}_{C}$, where $D = A/C$, and $U$ is a union of left cosets of $C$ in $A$. Let ${\cal L}$ be the set of all quotient groups of $A$ that are contained in $\Aut(\Haar(D,U))$.  For each $L\in{\cal L}$, let $\delta_L:L\to D$ be an $L$ to $D$ relabeling of $\Haar(D,U)$.  Let ${\cal S}$ be a solving set for $\Haar(D,U)$, and $T = \{s\delta_L:s\in {\cal S},L\in{\cal L}\}$.   Then a solving set for $\Haar(A,S)$ is $$\{t_A:t_A{\rm\ is\ any\ extension\ of\ }t{\rm\ to\ }\Z_2\times A{\rm\ and\ }t\in T\}.$$  
\end{lemma}

\begin{proof}
Set $\Gamma = \Haar(A,S)$.  As $\Gamma$ is connected, we see that ${\cal B}$ is a block system of $\Aut(\Gamma)$.  As $1 < C < A$ and $\Gamma\cong \Haar(D,U)\wr \bar{K}_{C}$, the action of $F_{A,S}$ on $B\in{\cal B}$ is not faithful.  Choosing $C$ to be maximal gives $\Haar(D,U)$ is irreducible by \cite[Lemma 1 (ii)]{Sabidussi1964}, and then $\Aut(\Haar(A,S))\cong\Aut(\Haar(D,U))\wr S_C$ by \cite[Theorem]{Sabidussi1959}.  Let $\phi\in S_{\Z_2\times A}$ such that $\phi(\Gamma) = \Haar(A,T)$ for some $T\subseteq A$.  As  $\Gamma$ is a wreath product, so is $\Haar(A,T)$ and $\Aut(\Haar(A,T))$ has a lexi-partition ${\cal F}$ with blocks of the same size as the the blocks of the lexi-partition of $\Gamma$ with respect to the empty graph on $C$.  We then have that ${\cal F}$ is the set of cosets of some subgroup of $A$ of size $\vert C\vert$.  So $\phi(\Gamma)/{\cal F}$ is a Haar graph $\Haar(E,U)$ for some quotient group $E$ of $A$ of order $\vert D\vert$ and $U\subseteq E$.  Thus $E\in{\cal L}$.  By \cite[Lemma 5]{DobsonM2005}, the lexi-partition of $\Gamma$ corresponding to $C$ and the lexi-partition ${\cal F}$ are the only block systems of $\Aut(\Gamma)$ and 
$\Aut(\phi(\Gamma))$, respectively, with blocks of size $\vert D\vert$.  Thus $\Haar(D,U)$ is isomorphic to $\Haar(E,U)$.  Then $\delta_E(\Haar(E,U)) = \Haar(D,V)$ for some $V\subseteq D$.  As ${\cal S}$ is a solving set for $\Haar(D,U)$, there is $s\in{\cal S}$ such that $s\delta_E(\Haar(E,U)) = \Haar(D,V)$.  It is now easy to see that $\Gamma$ and $\phi(\Gamma)$ are isomorphic by any extension of $s\delta_E$ to $\Z_2\times A$ (see \cite[Lemma 4.12]{BarberDpreprint} for a similar proof to this last fact).
\end{proof}


Let $A$ be an abelian group.  Note that if every Sylow subgroup of $A$ is elementary abelian or cyclic, then every quotient of $A$ also has elementary abelian or cyclic Sylow subgroups. In the case that the Sylow subgroup is cyclic, the quotients are equal.  If they are elementary abelian, they need not be equal, but can be made equal by applying an automorphism of $A$ to make the kernels of the action of $\widehat{A}_L$ on the lexi-partitions, and hence the quotients, the same.  As we may assume that solving sets always contain the automorphisms of $A$, we may then set every $\delta_L = 1$ in the statement.  This gives the following result.

\begin{lemma}\label{action unfaithful 3}
Let $A$ be an abelian group such that every Sylow subgroup is elementary abelian or cyclic, $S\subseteq A$ such that $\Haar(A,S)$ is connected, and there exists $C\le A$, chosen to be maximal, such that $\Haar(A,S)\cong \Haar(D,U)\wr \bar{K}_{C}$, where $D = A/C$, and $U$ is a union of left cosets of $C$ in $A$. Let $T$ be a solving set for $\Haar(D,U)$.  Then a solving set for $\Haar(A,S)$ is $$\{t_A:t_A{\rm\ is\ any\ extension\ of\ }t{\rm\ to\ }\Z_2\times A{\rm\ and\ }t\in T\}.$$  
\end{lemma}

We now turn to automorphisms of disconnected vertex-transitive Haar graphs.  It is somewhat ironic that determining isomorphism between these graphs will take the most work of the three cases of Theorem \ref{automorphism main result} that we will consider, as of course isomorphism between disconnected graphs obviously reduces to isomorphism between components.

The automorphism group of a disconnected vertex-transitive Haar graph of a group $G$ always has an imprimitive automorphism group with blocks the set of components of the graph, which are the left cosets of $\Z_2\times H$ for some $H\le G$.  As the graph is a wreath product, the ``natural'' group to consider is $(G/H)\times(\Z_2\times H)$, as this is the set permuted by the wreath product.  However, the group $\Z_2\ltimes\widehat{G}_L$ need not be such a direct product, so in the rest of this section we will consider the vertex set to be $(G/H)\times (\Z_2 \times H)$, but we will consider $\Z_2\ltimes\widehat{G}_L$ in its natural embedding in $S_{G/H}\wr (\Z_2\ltimes S_H)$ given by the Embedding Theorem \cite[Theorem 4.3.1]{Book}, which contains a natural subgroup isomorphic to $\Z_2\times G$ by \cite[Theorem 4.3.5]{Book}.  We begin with a couple of preliminary results, the first of which ``aligns'' the bipartition into a nice form.

\begin{lemma}\label{wreath components}
Let $A$ be an abelian group, and $\emptyset\not = S\subset A$ such that $\Haar(A,S)$ is disconnected.  Let $C$ be the component of $\Haar(A,S)$ that contains $(0,0)$, and let $a\in A$ such that $(1,-a)\in V(C)$.  Then $H = \la(a + S)(a + S)^{-1}\ra$ is a subgroup of $A$, and the components of $\Haar(A,a + S)$ are the left cosets of $\Z_2\times H$ in $\Z_2\times A$.  Consequently, $\Haar(A,a + S) = \bar{K}_m\wr\Haar(H,a + S)$ where $m = \vert A/H\vert$.
\end{lemma}

\begin{proof}
Let $\Gamma = \Haar(A,S)$.  As $S\not = \emptyset$, $\Gamma$ has edges.  Then each component of $\Gamma$ is bipartite, and its edges must have one endpoint in $B_0$ and one in $B_1$.  This implies that the set of edge endpoints in $V(C)\cap B_0$ and the set of edge endpoints in $V(C)\cap B_1$ is a bipartition of $C$.  As $\widehat{A}_L$ is transitive on $B_0$ and $B_1$, the set of vertices in a component contained in $B_i$ is a block system of $(\widehat{A}_L)^{B_i}$, and so is the set ${\cal C}_i$, $i \in \Z_2$, of cosets of some subgroup $H_i\le  A$ by \cite[Example 2.3.8]{Book}.   As $A$ is abelian, $\Aut(\Haar(A,S))$ contains the map $\iota:\Z_2\times A\to\Z_2\times A$ given by $\iota(i,j) = (i + 1,-j)$, and also interchanges $H_0$ and $H_1$. Thus $-H_1 = H_0$ and $H_1 = H_0$.  In $\bar{a}(\Gamma) = \Haar(A,a + S)$, the component that contains $(0,0)$ also contains $(1,0)$, and so has vertex set $\Z_2\times H$, where $H = H_0 = H_1$.  As $\widehat{A}_L\le\Aut(\Haar(A,a + S))$, the components of $\Haar(A,a + S)$ have vertex sets the left cosets of $\Z_2\times H$ in $\Z_2\times A$.  Finally, as the component of $\Haar(A, a + S)$ that contains $(0,0)$ is connected and is $\Haar(H,a + S)$, by \cite[Lemma 2.3 (iii)]{DuX2000} we have $H = \la (a + S)(a + S)^{-1}\ra$.  Thus $\Haar(A,a+S) = \bar{K}_m\wr\Haar(H,a + S)$.
\end{proof}

Once the components have been ``aligned'' as in the previous result, the isomorphism problem reduces to the isomorphism problem for the components.

\begin{lemma}\label{disconnected main}
Let $A$ be an abelian group, and $\emptyset\not = S\subset A$ such that $\Haar(A,S)$ is disconnected with $H\le A$ such that $\Haar(H,S)$ is a component of $\Haar(A,S)$.  Let $T\subseteq A$ such that $\Haar(A,T)$ is a disconnected Haar graph of $A$ with component $\Haar(H,T)$.  If $\Haar(A,S)$ and $\Haar(A,T)$ are isomorphic, then they are isomorphic by an element of $1_{S_{A/H}}\times {\cal S}$, where ${\cal S}$ is a solving set for $\Haar(H,S)$.
\end{lemma}

\begin{proof}
Let $\Gamma = \Haar(A,S)$.  As $\Haar(H,S)$ is a component of $\Gamma$, $\Haar(H,S)$ is connected and as $S\not = \emptyset$, $\Haar(H,S)$ has edges.  Also, $V(\Haar(H,S))$ contains $(0,0)$ and $(1,0)$.  Recall that we are considering $V(\Gamma) \le A/H\times\Z_2\times H$.  By Lemma \ref{wreath components}, the vertex sets of the components of $\Gamma$ are then the left cosets of $\Z_2\times H$ in $A/H\times\Z_2\times H$.

Let $\phi\in S_{A/H\times\Z_2\times H}$ such that $\phi(\Gamma) = \Haar(A,T)$. 
Then $\phi^{-1}\widehat{A}_L\phi\le\Aut(\Gamma)$ and the components of $\phi(\Gamma)$ are the left cosets of $\Z_2 \times H$ in $A/H\times\Z_2\times H$.  Let ${\cal C}$ be the set of components of $\Gamma$.  Then ${\cal C}$ is a block system of $\Aut(\Gamma)$ and, as the components of $\Gamma$ and $\phi(\Gamma)$ are the same, ${\cal C}$ is also a block system of $\Aut(\phi(\Gamma))$.  Hence $\phi\in S_{A/H}\wr S_{\Z_2\times H}$.  As $\Aut(\Gamma) = S_m\wr\Aut(\Haar(H,S))$ where $m = \vert A/H\vert$, we may assume without loss of generality that $\phi/{\cal C} = 1$, so that $\phi(x,y,z) = (x,\phi_x(y,z))$, where each $\phi_x\in S_{\Z_2\times H}$ and $x\in A/H$.  

As ${\cal S}$ is a solving set for $\Haar(H,S)$, there is $s\in {\cal S}$ such that $s(\Haar(H,S)) = \Haar(H,T)$.  Define $s_A:A/H\times\Z_2\times H\to A/H\times\Z_2\times H$ by $s_A(x,y,z) = (x,y,s(z))$.  Then $s_A\in 1_{S_{A/H}}\times{\cal S}$.  As $\widehat{H}_L\le s\Aut(\Haar(H,S))s^{-1}\le\Aut(\Haar(H,T))$ and $\Aut(\phi(\Gamma)) = S_m\wr\Aut(\Haar(H,T))$, by \cite[Theorem 4.3.5]{Book} we see that $\widehat{A}_L/{\cal C}\wr (\Z_2\times H_L)\ge \widehat{A}_L$ and $s_A(\Gamma)$ is a Haar graph of $A$.  As $s(S) = T$, $s_A(\Gamma) = \Haar(A,T)$.  As $s_A\in 1_{S_{A/H}}\times {\cal S}$, the result follows.
\end{proof}

The next example gives a sufficient condition for a group $G$ to not be a BCI-group.

\begin{example}\label{disconnected bad}
Let $G$ be a group, and $H_1,H_2 < G$ such that $\vert H_1\vert = \vert H_2\vert$ and there exists no $\alpha\in\Aut(G)$ such that $\alpha(H_1) = H_2$.  Then $\Haar(G,H_1)\cong\Haar(G,H_2)$, but there is no element of $\Iso(G)$ that maps $\Haar(G,H_1)$ to $\Haar(G,H_2)$, nor is there an element of $\Iso(H_1)$ that maps the component of $\Haar(G,H_1)$ that contains $(0,0)$ to the component of $\Haar(G,H_2)$ that contains $(0,0)$.  
\end{example}

\begin{proof}
The Haar graphs $\Haar(G,H_1)$ and $\Haar(G,H_2)$ are disconnected by \cite[Lemma 2.3 (iii)]{DuX2000} and each component of each graph is isomorphic to $K_{m,m}$, where $m = \vert H_1\vert = \vert H_2\vert$.  Hence $\Haar(G,H_1)$ and $\Haar(G,H_2)$ are isomorphic.  Note that $\tau$ (as in Lemma \ref{normalizer}) and every element of $\bar{G}$ maps $\Z_2\times(G/H_1)$ to itself.  As there is no automorphism of $G$ that maps $H_1$ to $H_2$, by Lemma \ref{biCayley Iso}, no element of $\Iso(G)$ maps $\Haar(G,H_1)$ to $\Haar(G,H_2)$. Note that the component of $\Haar(G,H_1)$ that contains $(0,1_G)$ has vertex set $\Z_2\times H_1$, while the component of $\Haar(G,H_2)$ that contains $(0,1_G)$ has vertex set $\Z_2\times H_2$.  As every element of $\Iso(H_1)$ is a permutation of $\Z_2\times H_1$, the result follows. 
\end{proof}

We now turn to consequences of our previous results for abelian groups.

\begin{corollary}
Let $A$ be an abelian $\ABCI$-group.  Then every Sylow subgroup of $A$ is elementary abelian or cyclic.
\end{corollary}

\begin{proof}
We observe that if $A$ is an abelian group such that a Sylow $p$-subgroup $P$ of $A$ is not cyclic or elementary abelian for some prime divisor $p$ of $\vert A\vert$, then $P$ contains subgroups isomorphic to both $\Z_{p^2}$ and $\Z_p\times\Z_p$.  Of course, no automorphism of any group maps a subgroup $H$ to a nonisomorphic subgroup $K$.  The result follows by Example \ref{disconnected bad}.
\end{proof}

Our next result gives isomorphisms between disconnected Haar graphs of abelian groups in terms of a solving set of the component that contains $(0,0)$.

\begin{corollary}\label{disconnected final}
Let $A$ be an abelian group and $\emptyset\not = S\subset A$ such that $\Haar(A,S)$ is disconnected.  Let $a\in A$ such that $(1,-a)$ is in the same component $C$ of $\Haar(A,S)$ that contains $(0,0)$, and $H = \la(a + S)(a + S)^{-1}\ra$.  The components of $\Haar(A,a + S)$ are the left cosets of $\Z_2\times H$ in $\Z_2\times A$.  Let ${\cal K}$ be the set of subgroups $K$ of $A$ of order $\vert H\vert$ such that $C$ is $K$-semiregular with the orbits of $K$ the bipartition of $C$.  For each $K\in{\cal K}$, let $\omega_K$ be a $K$ to $H$ relabeling of $C$, and $\delta_K:\Z_2\times A\to\Z_2\times A$ that is $\omega_K^{-1}$ on each left coset of $\Z_2\times K$ and maps the set of left cosets of $\Z_2\times K$ in $\Z_2\times A$ to the set of left cosets of $\Z_2\times H$ in $\Z_2\times A$.  An $\ABCI(A,S)$ is $$\{\overline{-b}\widehat{\delta_K}^{-1}t:b\in A, K\in{\cal K}{\rm\ and\ }t\in ((1\times A/H)\times\ABCI(H,a + S))\overline{a}\}.$$  
\end{corollary}

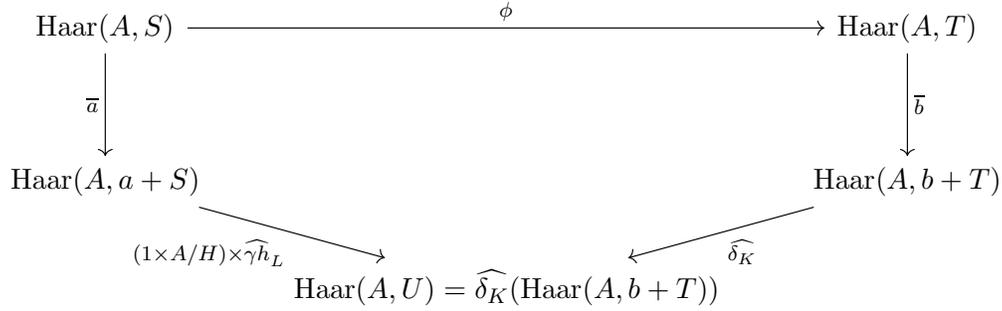
\begin{figure}
    \begin{center}
    \begin{tikzcd}
        \Haar(A,S) \arrow[dd, swap, "\overline{a}"] \arrow[rr, "\phi"] & & \Haar(A,T) \arrow[dd, "\overline{b}"]\\ \\
        \Haar(A,a+S) \arrow[dr, swap, "(1 \times A/H) \times \widehat{\gamma h}_L"] & & \Haar(A,b+T) \arrow[dl, "\widehat{\delta_{K}}"] \\
        & \Haar(A,U) = \widehat{\delta_{K}}(\Haar(A,b+T))  
    \end{tikzcd}
    \end{center}
    \caption{The maps in Corollary \ref{disconnected final}}
    \label{maps}
\end{figure}

\begin{proof}
This proof contains quite a few maps and Haar graphs.  We have included Figure \ref{maps} to help the reader keep them all straight. The information about the components of $\Haar(A, a + S)$ comes from Lemma \ref{wreath components}.  Let $\phi\in S_{\Z_2\times A}$ such that $\phi^{-1}\widehat{A}_L\phi\le\Aut(\Haar(A,S))$.  Then $\phi(\Haar(A,S)) = \Haar(A,T)$ for some $T\subset A$ is a Haar graph of $A$ isomorphic to $\Haar(A,S)$.  Also by Lemma \ref{wreath components}, there is $b\in A$ such that the components of $\phi(\Haar(A,b + T))$ are the left cosets of $\Z_2\times K$ in $\Z_2\times A$, where $K\le A$.  

Next, we show that $\delta_K$ as in the statement exists.  Let ${\cal C}$ be the block system of $\la\tau,\widehat{A}_L\ra$ that is the set of left cosets of $\Z_2\times K$ in $\Z_2\times A$.  As $\vert K\vert = \vert H\vert$, we have $\vert(\Z_2\times A)/(\Z_2\times H)\vert = \vert(\Z_2\times A)/(\Z_2\times K)\vert$, so there is a bijection $\alpha:(\Z_2\times A)/(\Z_2\times K) \to(\Z_2\times A)/(\Z_2\times H)$.  
As the vertex sets of the components of $\Haar(A,b + T)$ are the block system ${\cal C}$, we see that $\fix_{\Aut(\Haar(A,b+T))}({\cal C})^C$ contains a semiregular subgroup with two orbits isomorphic to $H$.  Let $\delta$ be a $K$ to $H$ relabeling of $\Haar(K,b+T)$.  Recall that we are considering $\Z_2\ltimes\widehat{A}_L$ in its natural embedding in $S_{A/H}\wr (\Z_2\ltimes S_H)$ given by the Embedding Theorem \cite[Theorem 4.3.1]{Book}.  Similarly, we consider $\Z_2\ltimes\widehat{A}_L$ as in its natural embedding of $S_{A/K}\wr(\Z_2\ltimes S_K)$ given by the Embedding Theorem.  Then $\delta_K:\Z_2\times A\to \Z_2\times A$ given by $\delta_K(i,j) = (\alpha(i),\delta(j))$ is a bijection from $\Z_2\times A$ to $\Z_2\times A$ that is a $K$ to $H$ relabeling on each left coset of $\Z_2\times K$ in $\Z_2\times A$ and maps the set of left cosets of $\Z_2\times K$ in $\Z_2\times A$ to the set of left cosets of $\Z_2\times H$ in $\Z_2\times A$.

As $\Haar(A,S)$ is disconnected, its automorphism group contains $S_{A/H}\wr (\Z_2\times H_L)$.  As we consider $\Z_2\ltimes\widehat{A}_L$ in its natural embedding in $S_{A/H}\wr (\Z_2\ltimes S_H)$ given by the Embedding Theorem, $S_{A/H}\wr (\Z_2\times H_L)$ contains $\Z_2\times\widehat{A}_L$.  
Thus $\widehat{\delta_K}\overline{b}\phi(\Haar(A,S))$ is a Haar graph $\Haar(A,U)$ for some connection set $U\subset A$. As the vertex sets of the components of $\Haar(A,a + S)$ and $\Haar(A,U)$ are the same, we see that $\Haar(H,a + S)$ and $\Haar(H,U)$ are isomorphic by a map of the form $\gamma h$, where $\gamma\in \Iso(H)$ and $h\in \ABCI(H,S)$.  Again, as the vertex sets of the components of $\Haar(A,a + S)$ and $\Haar(A,U)$ are the same, $\Haar(A,a + S)$ and $\Haar(A,U)$ are isomorphic by a map of the form $(1\times A/H)\times{\gamma h}$ (this is the map that is $\gamma h$ on each component $\Haar(A,S)$, which are also the components of $\Haar(A,U)$).  Hence $\Haar(A,S)$ and $\Haar(A,T)$ are isomorphic by a map of the form $\overline{-b}\widehat{\delta_K}^{-1}((1\times A/H)\times \gamma h)\overline{a}$, and an $\ABCI(A,S)$ is $\{\overline{-b}\widehat{\delta_K}^{-1} t:K\in{\cal K}{\rm\ and\ }t\in ((1\times A/H)\times\ABCI(H,a + S))\overline{a}\}$.  
\end{proof}

In the next section, we will be interested in Haar graphs of $A$ for which every Sylow subgroup is elementary abelian or cyclic.  The next result gives the specific version of the previous result for these groups.

\begin{corollary}\label{disconnected cyclic and ea}
Let $A$ be an abelian group such that every Sylow subgroup of $A$ is elementary abelian or cyclic.  Let $\emptyset \not =  S\subset A$ such that $\Haar(A,S)$ is disconnected, and $a\in A$ such that $(1,-a)$ is in the same component $C$ of $\Haar(A,S)$ that contains $(0,0)$, and $H = \la (a + S)(a + S)^{-1}\ra$.  Then $H\le A$ and an $\ABCI(A,S)$ is $((1\times A/H)\times\ABCI(H,a + S))\overline{a}$.
\end{corollary}

\begin{proof}
Let $\Gamma = \Haar(A,S)$.  Note that as $S\not = \emptyset$, the graph $\Gamma$ has edges.  In $\bar{a}(\Gamma) = \Haar(A,a + S)$, the component $\bar{a}(C)$ that contains $(0,0)$ also contains $(1,0)$.  By Lemma \ref{wreath components}, we see that $\Haar(A,a + S) = \bar{K}_m\wr\Haar(H,a + S)$, where $m$ is the number of components of $\Haar(A,a + S)$, and $\Z_2\times H = V(\bar{a}(C))$.  Also, the vertex sets of the components are the cosets of $\Z_2\times H$ in $\Z_2\times A$.  By Lemma \ref{disconnected main}, an $\ABCI(A,a+S)$ is $\{\overline{-b}\widehat{\delta_K}^{-1} t:K\in{\cal K}{\rm\ and\ }t\in ((1\times A/H)\times\ABCI(H,a + S))\overline{a}\}$.

Let $P$ be a Sylow $p$-subgroup of $H$ for some prime $p$.  If a Sylow $p$-subgroup $P$ of $A$ is cyclic, then $P$ is the unique subgroup of $A$ of order $\vert P\vert$.  Hence any subgroup of $A$ whose order is divisible by $\vert P\vert$ contains $P$.  If a Sylow $p$-subgroup of $A$ is elementary abelian and $K$ is any subgroup of $A$ of order $\vert H\vert$, then $K$ has an elementary abelian Sylow $p$-subgroup, and so there exists an automorphism of $A$ which maps a Sylow $p$-subgroup of $K$ to $P$.  We conclude that each map $\delta_{K}$ as defined in Corollary \ref{disconnected final} is an automorphism of $A$.  Then $\bar{b}\widehat{\delta_K}\in\Iso(A,S)$, an $\ABCI(A,a+S)$ is $(1\times A/H)\times\ABCI(a+S)$, and so an $\ABCI(A,S)$ is $((1\times A/H)\times\ABCI(H,a + S))\overline{a}$.  
\end{proof}

We next consider the $\ABCI$-problem for some abelian groups which have a nice property given in the next definition.

\begin{definition}
A group $G$ is called a {\bf homogeneous group} if each isomorphism between any two isomorphic subgroups of $G$ extends to an automorphism of $G$.
\end{definition}

Li \cite{Li1999} has classified finite homogeneous groups.  In particular, a finite abelian group is homocyclic (that is, a direct product of isomorphic cyclic groups) provided that all of its Sylow subgroups are homocyclic.  

\begin{corollary}\label{disconnected ABCI}
Let $A$ be an abelian homogeneous group so that every Sylow subgroup of $A$ is homocyclic. Then any disconnected Haar graph $\Haar(A,S)$ of $A$ with components isomorphic to an ABCI-graph of a subgroup $H\le A$ is an ABCI-graph of $A$.   
\end{corollary}

\begin{proof}
Let $\Gamma = \Haar(A,S)$, and $C$ be the component of $\Gamma$ that contains $(0,0)$.  Let $(1,-a)\in V(C)$ for some $a\in A$.  In $\bar{a}(\Gamma) = \Haar(A,a + S)$, the component $\bar{a}(C)$ that contains $(0,0)$ also contains $(1,0)$.  By Lemma \ref{wreath components}, we see that $\Haar(A,a + S) = \bar{K}_m\wr\Haar(H,a + S)$ where $H = \la (a+S)(a+S)^{-1} \ra$.  Here, $m$ is the number of components of $\Haar(A,a + S)$, and $\Z_2\times H = V(\bar{a}(C))$.  Also, the set of components of $\Haar(A,a + S)$ is the set of cosets of $\Z_2\times H$ in $\Z_2\times A$ as $\bar{a}(C)$ contains both $(0,0)$ and $(1,0)$.  By Lemma \ref{disconnected main}, an $\ABCI(A,a+S)$ is $1_{S_{A/H}}\times\ABCI(H,a+S)$.  As $a\in C$ and $\bar{a}\in \Iso(A,a + S)\cap\Iso(A,S)$, we may assume without loss of generality that $a = 0$.  As by hypothesis $C$ is an $\ABCI$-graph of $H$, by Lemma \ref{abelian Haar iso}, we see that $\Haar(H,S)$ is isomorphic to another Haar graph of $H$ if and only if they are isomorphic by a map of the form $\widehat{\beta}\bar{h}$, where $\beta\in\Aut(H)$ and $h\in H$.  As an abelian group is homogeneous if and only if every Sylow subgroup is homocyclic \cite{Li1999}, the automorphism $\beta\in\Aut(H)$ extends to an automorphism $\alpha_{\beta}\in\Aut(A)$.  Of course, $h\in A$, so the map $\bar{h}$ with domain $\Z_2\times H$ extends to the map $\bar{h}$ with domain $\Z_2\times A$.  Now suppose $T\subset A$ such that $\Haar(A,S)\cong \Haar(A,T)$.  Then there exists $b\in A$ such that the component of $\overline{b}(\Haar(A,T))$ that contains $(0,0)$ also contains $(1,0)$.  As above, we may assume without loss of generality that $b = 0$.  Then $\widehat{\alpha}\bar{h}$ maps $C$ to the component of $\Haar(A,T)$ that contains $(0,0)$ and $(1,0)$.  This implies that the components of $\Haar(A,S)$ and $\Haar(A,T)$ are the same, and so $\widehat{\alpha_{\beta}}\bar{h}(\Haar(A,S)) = \Haar(A,T)$.
\end{proof}

\section{Applications}\label{applications}

First, Babai has shown that the dihedral group $D_p$ of order $2p$ is a CI-group with respect to ternary relational structures \cite[Theorem 4.4]{Babai1977}, and so for binary relational structures, and so for graphs. Also, the automorphism group of every Haar graph of a cyclic group contains a regular dihedral subgroup, and so is also a Cayley graph of a dihedral group by \cite[Lemma 4]{Sabidussi1958} or \cite[Theorem 1.2.20]{Book}.  This shows that $\Z_p$ is an ABCI-group, as it is easy to check that all automorphisms of $D_p$ are contained in $\Iso(\Z_p)$.  The solution to the isomorphism problem for Cayley digraphs of some other dihedral groups $D_n$ for special $n$ was given in \cite[Theorem 22]{Dobson2002}.   Similarly, automorphism groups of every Cayley digraph of the dihedral group $D_p$ of order $2p$ are given in \cite[Theorem 3.2]{Dobson2006a}, and so all automorphism groups of all Haar graphs of the group $\Z_p$ are given there as well.  

\begin{definition}
The {\bf canonical double cover of a graph $\Gamma$} is the bipartite graph $B\Gamma$ with $V(B\Gamma) = \Z_2\times V(\Gamma)$ with edge set $\{(0,v)(1,w):vw\in E(\Gamma)\}$.
\end{definition}

If $\Gamma = \Haar(G,S)$, then $(0,g)(1,sg)\in E(\Gamma)$ if and only if $s\in G$, and so if $\Cay(G,S)$ is a graph, $\Haar(G,S)$ is the canonical double cover of $\Cay(G,S)$.  

There has been much interest in graphs $\Gamma$ for which the canonical double cover of a graph has automorphism group $\Z_2\times \Aut(\Gamma)$ \cite{MarusicSS1989,Wilson2008,QinXZ2019,HujdurovicMM2021,HujdurovicMM2021,FernandezH2022,HujdurovicMM2023,HujdurovicM2024}.  It is straightforward to show that bipartite graphs, graphs with twins, and disconnected graphs do not have automorphism groups $\Z_2\times\Aut(\Gamma)$.  This has led to the following definition.

\begin{definition}
Let $\Gamma$ be a graph.  We say that $\Gamma$ is {\bf unstable} if $\Aut(B\Gamma)\not = \Z_2\times\Aut(\Gamma)$.  We call $\Gamma$ {\bf trivially unstable} if $\Gamma$ is connected, not bipartite, twin-free, and unstable. 
\end{definition}

The authors are unaware of an analogue of the canonical double cover for digraphs.  However, as Haar graphs can be constructed using connection sets that give Cayley digraphs that are not graphs, we can think of Haar graphs with connection sets that give Cayley digraphs that are not graphs as such analogues.  As the idea behind a stable graph is that its automorphism group gives the automorphism group of its canonical double cover, Theorem \ref{automorphism main result} motivates the following definition.

\begin{definition}
Let $A$ be an abelian group and $S\subseteq G$.  We say that $\Cay(A,S)$ is {\bf unstable} if $$\Aut(\Gamma) \not = \bar{a}^{-1}\Z_2\ltimes\Aut(\Cay(A,a + S))\bar{a}$$ for some $a\in A$.  We say that $\Cay(A,S)$ is {\bf nontrivially unstable} if it is connected, not biparite, twin-free, and unstable.
\end{definition}

We remark that we are only defining nonstable and nontrivially unstable for Cayley digraphs of abelian groups as it is not clear what the appropriate definition should be for nonabelian groups as at this time there is no analogue of Theorem \ref{automorphism main result} for nonabelian groups, and the appropriate definition for nonabelian groups should be based on such a result.

\begin{problem}
Classify all nontrivially unstable Cayley digraphs of abelian groups.
\end{problem}

The next result shows that the previous definition of a nontrivially unstable Cayley graph is the same in both definitions above.  This is a special case of Dave Morris's Theorem \ref{Daves result} stated below.  We give a self-contained proof as it is not difficult.

\begin{lemma}\label{a = 0}
Let $A$ be an abelian group of odd order and $S\subseteq A$ such that $S = -S$ and $$\Aut(\Haar(A,S)) = \bar{a}^{-1}\Z_2\ltimes\Aut(\Cay(A,S))\bar{a}$$ for some $a\in A$.  Then $a = 0$ and $\Aut(\Haar(A,S)) = \Z_2\times\Aut(\Cay(A,S))$.
\end{lemma}

\begin{proof}
As $S = -S$, $\Cay(A,S)$ is a graph, and so contains the map $\iota:A\to A$ by $\iota(a) = -a$.  The map $\widehat{\iota}$ is then contained in $\Aut(\Haar(A,S))$ by \cite[Lemma 1.4]{Dobson2022}.  Then $\widehat{\iota}$ and $\bar{a}^{-1}\widehat{\iota}\bar{a}$ are contained in $\Aut(\Haar(A,S))$ and $\widehat{\iota}\bar{a}^{-1}\widehat{\iota}\bar{a}(0,j) = (0,j)$ while $\widehat{\iota}\bar{a}^{-1}\widehat{\iota}\bar{a}(1,j) = (1,j + 2a)$.  As $A$ has odd order, $2a$ is not the identity if $a\not = 0$.  Suppose $a\not = 0$.  Note that as $\Aut(\Haar(A,S)) = \bar{a}^{-1}\Z_2\ltimes\Aut(\Cay(A,S))\bar{a}$, $\Haar(A,S)$ is connected, so ${\cal B}$ is a block system of $\Aut(\Haar(A,S))$. Then the action of $\widehat{\iota}\bar{a}^{-1}\widehat{\iota}\bar{a}$ on $B_0\in{\cal B}$ with $(0,0)\in B_0$ is not faithful.  However, by \cite[Theorem 2.8]{Dobson2022} we see that $\Aut(\Haar(A,S))$ is a nontrivial wreath product, a contradiction.  Thus $a = 0$.  

Finally, we observe that the map $\tau:\Z_2\times A\to\Z_2\times A$ given by $\tau(i,j) = (i+1,-j)$ is always contained in $\Aut(\Haar(A,S))$ (this was implicitly shown in \cite{HladnikMP2002}).  Then $\tau\bar{\iota}(i,j) = (i + 1, j)$ is contained in $\Aut(\Haar(A,S))$.  This then implies that $\Aut(\Haar(A,S)) = \Z_2\times\Aut(\Cay(A,S))$, and the result follows. 
\end{proof}

The next result gives a sufficient condition for the isomorphism problem for Haar graphs of an abelian group $A$ to reduce to the isomorphism problem for either ``smaller'' Haar graphs obtained as a component or quotient, or to the isomorphism problem for its corresponding Cayley graph.  

\begin{theorem}\label{iso main}
Let $A$ be an abelian group, and $S\subseteq A$.  If $\Cay(A,S)$ is not nontrivally unstable, then the isomorphism problem for $\Gamma = \Haar(A,S)$ can be determined from the solution to the isomorphism problem for a nontrivial block quotient graph of $\Gamma$, a component of $\bar{a}(\Gamma)$, where $(1,-a)$ is in the same component of $\Gamma$ as $(0,0)$, or the isomorphism problem for $\Cay(A,a+S)$, where $a\in A$.
\end{theorem}



\begin{proof}
We consider the four cases given by Theorem \ref{abelian Haar iso} separately.  Suppose $\Gamma$ is disconnected and let $\Haar(A,T)\cong \Gamma$.  Let $(1,-a)$ be in the component $C$ of $\Haar(A,T)$ that contains $(0,0)$.  By Corollary \ref{disconnected final}, the isomorphism problem for $\Gamma$ reduces to the isomorphism problem for $\bar{a}(C)$, a subgraph of $\bar{\Gamma}$.  So we assume without loss of generality that $\Gamma$ is connected. If $\Gamma$ is a nontrivial wreath product with an empty graph, then there exists $1\not = C\le A$, chosen to be maximal, such that $\Haar(A,S)\cong \Haar(D,U)\wr \bar{K}_{C}$, where $D = A/C$, and $U$ is a union of left cosets of $C$ in $A$.  Let ${\cal C}$ be the left cosets of $C$ in $\Z_2\times A$.  Let ${\cal L}$ be the set of all quotient groups of $A$ that are contained in $\Aut(\Haar(D,U))$.  For each $L\in{\cal L}$, let $\delta_L:L\to D$ be an $L$ to $D$ relabeling of $\Haar(D,U)$.  Let ${\cal S}$ be a solving set for $\Haar(D,U)$, and $T = \{s\delta_L:s\in {\cal S},L\in{\cal L}\}$.   By Lemma \ref{action unfaithful 1} a solving set for $\Gamma$ is $$\{t_A:t_A{\rm\ is\ any\ extension\ of\ }t{\rm\ to\ }\Z_2\times A{\rm\ and\ }t\in T\}.$$
Thus the isomorphism problem for $\Gamma$ can be determined from the solution to the isomorphism problem for a nontrivial block quotient graph of $\Gamma$.  By Theorem \ref{automorphism main result}, the only remaining possibility is that $\Aut(\Gamma)\cong \bar{a}^{-1}\Z_2\ltimes\Aut(\Cay(A,a + S))\bar{a}$ for some $a\in A$, and by Lemma \ref{Reduction in Case 3}, an $\ABCI(A,a + S)$ is $\widehat{T}\bar{a}$, where $T$ is a $\CI(G,a +S)$.  Then the isomorphism problem for $\Gamma$ can be determined from the solution to the isomorphism problem for $\Cay(A,a+S)$.
\end{proof}

Note that for the smaller Haar graphs of, say, the group $B$ obtained as a quotient or subgroup of $A$ for some $\Haar(A,S)$, we do NOT mean the isomorphism problem among Haar graphs of only the group $B$ as in both cases the isomorphic image may be a Haar graph of some other group, say $C$, under an isomorphism.  Thus we mean that it reduces the problem to the isomorphism problem among all Haar graphs isomorphic to a Haar graph of $B$.  To correct this problem, one would need to know the maps that are $C$ to $B$ relabelings.  But we remind the reader that all such relabelings are isomorphisms if every Sylow subgroup of $A$ is elementary abelian or cyclic, and so for such $A$ we reduce to the isomorphism problem for Haar graphs of the group $B$.

In the case where the isomorphism problem for $\Haar(A,S)$ reduces to the isomorphism problem for $\Cay(A,a + S)$ in Theorem \ref{iso main}, the value of $a$ is given in \cite[Theorem 2.12]{Dobson2022}.

Morris proved the following result \cite[Theorem 1.1]{Morris2021}. 

\begin{theorem}\label{Daves result}
Let $A$ be an abelian group of odd order, and $S\subseteq A$ such that $S = -S$.  If $\Cay(A,S)$ is connected and twin-free, then $\Aut(\Haar(A,S)) = \Z_2\times\Aut(\Cay(A,S))$.
\end{theorem}

For us, Morris's result shows that the sufficient condition given in Theorem \ref{iso main} holds for Haar graphs of odd order for which $S = -S$.  This gives the following result.

\begin{corollary}
Let $A$ be an abelian group of odd order, and $S\subseteq A$ such that $S = -S$.  Then the solution to the isomorphism problem for $\Haar(A,S)$ can be determined from the solution to the isomorphism problem for a nontrivial block quotient graph, component, or the isomorphism problem for $\Cay(A,S)$.
\end{corollary}

We remark that if one is only interested in the previous result, it can be obtained from the automorphism group of the canonical double cover of Cayley graphs of cyclic groups of odd order, which are given in \cite{Morris2021}.

The isomorphism problem for Cayley graphs of abelian groups has received considerable attention, but the number of groups for which the isomorphism problem has been solved is somewhat limited.  However, the isomorphism problem has been solved for Cayley graphs of cyclic groups \cite{Muzychuk2004}, giving the following result.

\begin{corollary}\label{cyclic solution}
Let $A$ be a cyclic group of odd order.  Let $S\subseteq A$ such that $S = -S$.  Then a solving set for $\Haar(A,S)$ is known.
\end{corollary}

The isomorphism problem has been solved for Cayley digraphs of the following noncyclic abelian groups of odd order:  $\Z_p^k$, $1\le k\le 5$, \cite{Babai1977,Godsil1983,Dobson1995,Morris2015,HirasakaM2001,FengK2018}, $\Z_p\times\Z_{p^2}$ \cite{Dobson2000}, $\Z_p\times\Z_q^2$, \cite{KovacsM2009}, $\Z_p\times\Z_q^3$ \cite{SomlaiM2021}, $\Z_p^2\times\Z_q^2$ and $\Z_p^2\times\Z_n$ \cite{KovacsMPPRS2023}, where $q$ and $p$ are distinct primes and $n$ is a square-free integer.  Also, the isomorphism problem has been solved for Cayley digraphs of some groups, satisfying certain arithmetic conditions, such that every Sylow subgroup is elementary abelian of order at most $p^5$ or is cyclic \cite{Dobson2018a}.  The isomorphism problem for Haar graphs with self-inverse connection sets can now be solved from the isomorphism problem for Cayley graphs for the groups given above using results in this paper (we omit the various statements for all of these groups) and the papers cited earlier in this paragraph, with the exception of $\Z_p\times\Z_{p^2}$.  For $\Z_p\times\Z_{p^2}$ all that is missing is appropriate $\Z_p^3$ to $\Z_p\times\Z_{p^2}$ labelings.

While Theorem \ref{iso main} in a very real sense solves the isomorphism problem for Haar graphs of abelian groups of odd order whose connection set is self-inverse, it does not reduce the ABCI problem to the CI problem.  There are two obstacles to this, with the obstacles coming from graphs that have twins or are disconnected.  What can now be done is the next result.

\begin{corollary}
Let $A$ be an abelian group of odd order such that every Sylow subgroup of $A$ is elementary abelian or cyclic.  Let $S\subseteq A$ such that $S = -S$ and $\Cay(A,S)$ is a CI-graph. Then $\Haar(A,S)$ is an ABCI-graph. 
\end{corollary}

\begin{proof}
By Theorem \ref{Daves result} we need only consider when $\Haar(A,S)$ has automorphism group $\Z_2\times\Aut(\Cay(A,S))$, $\Cay(A,S)$ has a twin, or $\Cay(A,S)$ is disconnected.  If $\Aut(\Haar(A,S)) = \Z_2\times\Aut(\Cay(A,S))$ and $\Cay(A,S)$ is a CI-graph, then by Lemma \ref{Stable CI} we have that $\Haar(A,S)$ is an $\ABCI$-graph.  

If $\Cay(A,S)$ is connected and has a twin, then $\Haar(A,S)$ has a twin and is connected as $\vert A\vert$ is odd (and so $\Cay(A,S)$ is not bipartite). Thus $\Haar(A,S)$ satisfies Theorem \ref{automorphism main result} (2).  Then there exists $H\le A$, chosen to be maximal, such that $\Haar(A,S)\cong \Haar(D,U)\wr \bar{K}_{H}$, where $D = A/H$, and $U$ is a union of left cosets of $H$ in $A$.  As $H$ is maximal, $\Haar(D,U)$ has no twins, and as it is the block quotient of a connected graph, it is connected.  As $A$ has odd order, $D$ has odd order, and we see by Theorem \ref{Daves result} that $\Aut(\Haar(D,U)) = \Z_2\times\Aut(\Cay(D,U))$, and $\Cay(D,U)$ is not unstable.  As by \cite{DobsonM2015a} the quotient of a CI-graph is a CI-graph, we see that $\Cay(D,U)$ is a CI-graph. Hence a $\CI(D,U)$ is $\Aut(D)$.  By Lemma \ref{action unfaithful 3}, an $\ABCI(D,U)$ is $T = \widehat{\Aut(D)}$.  As every Sylow subgroup of $A$ is elementary abelian or cyclic, we may choose an $\ABCI(A,S)$ to be $T_A\le\Aut(A)$.  Then $\ABCI(A,S)\le\widehat{\Aut(A)}\le\Iso(A,S)$, and $\Haar(A,S)$ is an ABCI-graph by definition.

If $\Cay(A,S)$ is disconnected, by Corollary \ref{disconnected cyclic and ea} an $\ABCI(A,S)$ is $((1\times A/H)\times\ABCI(H,a + S))\overline{a}$, where $(1,-a)$ is in the component of $\Haar(A,S)$ that contains $(0,0)$.  By Lemma \ref{a = 0}, we have $a = 0$ and so $\overline{a} = 1$.  Hence an $\ABCI(A,S)$ is $(1\times A/H)\times\ABCI(H,S)$.  If $\Haar(H,S)$ is an $\ABCI$-graph, then as $A$ is homogeneous, the result follows by Corollary \ref{disconnected ABCI}.  If $\Haar(H,S)$ does not have a twin, then by arguments in the first paragraph, we have that $\Haar(H,S)$ is an ABCI-graph.  If $\Haar(H,S)$ has a twin, then we showed in the previous paragraph that a connected Haar graph of $H$ of a CI-graph with a twin is an $\ABCI$-graph.  Thus in all cases where $\Cay(A,S)$ is disconnected, we have that $\Haar(A,S)$ is an $\ABCI$-graph.   
\end{proof}

\begin{corollary}
Let $A$ be an abelian group of odd order such that every Sylow subgroup of $A$ is elementary abelian.  If $A$ is a CI-group with respect to digraphs and $S\subseteq A$ such that $S = -S$, then $\Haar(A,S)$ is an ABCI-graph.
\end{corollary}

\bibliography{References}{}
\bibliographystyle{amsplain}

\end{document}